\documentclass[11pt]{amsart}
\usepackage{amscd,verbatim,amssymb}
\usepackage[colorlinks,linkcolor=blue,citecolor=blue,urlcolor=red]{hyperref}

\newcommand{\Pic}{\operatorname{Pic}}
\newcommand{\NS}{\operatorname{NS}}
\newcommand{\Hom}{\operatorname{Hom}}
\newcommand{\uHom}{\underline{\operatorname{Hom}}}
\newcommand{\End}{\operatorname{End}}
\newcommand{\Ext}{\operatorname{Ext}}

\newcommand{\Ker}{\operatorname{Ker}}
\newcommand{\IM}{\operatorname{Im}}
\newcommand{\Coker}{\operatorname{Coker}}
\newcommand{\Corr}{\operatorname{Corr}}
\newcommand{\oCorr}{\overline{\Corr}}
\newcommand{\oCH}{\operatorname{\overline{CH}}}
\newcommand{\Alb}{\operatorname{Alb}}
\newcommand{\LAlb}{\operatorname{LAlb}}
\newcommand{\RPic}{\operatorname{RPic}}
\newcommand{\Ab}{\operatorname{\mathbf{Ab}}}
\newcommand{\Abs}{\operatorname{\mathbf{AbS}}}
\newcommand{\op}{{\operatorname{o}}}
\newcommand{\cl}{\operatorname{cl}}
\newcommand{\Spec}{\operatorname{Spec}}
\newcommand{\car}{\operatorname{char}}

\newcommand{\gr}{\operatorname{gr}}
\newcommand{\cd}{\operatorname{cd}}
\newcommand{\ab}{{\operatorname{ab}}}

\newcommand{\sA}{\mathcal{A}}
\newcommand{\sB}{\mathcal{B}}
\newcommand{\sC}{\mathcal{C}}
\newcommand{\sE}{\mathcal{E}}
\newcommand{\sF}{\mathcal{F}}
\newcommand{\sH}{\mathcal{H}}
\newcommand{\sM}{\mathcal{M}}

\newcommand{\A}{\mathbf{A}}
\newcommand{\C}{\mathbf{C}}
\newcommand{\G}{\mathbb{G}}
\renewcommand{\L}{\mathbb{L}}
\newcommand{\Q}{\mathbf{Q}}
\newcommand{\Z}{\mathbf{Z}}
\newcommand{\un}{\mathbf{1}}
\newcommand{\by}{\xrightarrow}

\newcommand{\iso}{\by{\sim}}

\newcommand{\inj}{\hookrightarrow}

\newcommand{\surj}{\rightarrow\!\!\!\!\!\rightarrow}
\newcommand{\Surj}{\relbar\joinrel\surj} 
\newcommand{\et}{{\mathrm{\acute{e}t}}}

\newcommand{\cont}{{\mathrm{cont}}}
\newcommand{\tors}{{\mathrm{tors}}}

\renewcommand{\phi}{\varphi}
\renewcommand{\epsilon}{\varepsilon}

\newcounter{spec}
\newenvironment{thlist}{\begin{list}{\rm{(\roman{spec})}}%
{\usecounter{spec}\labelwidth=20pt\itemindent=0pt\labelsep=10pt}}%
{\end{list}}%

\numberwithin{equation}{section}

\newtheorem{thm}{Theorem}[section]
\newtheorem{lemma}[thm]{Lemma}
\newtheorem{prop}[thm]{Proposition}
\newtheorem{cor}[thm]{Corollary}
\theoremstyle{definition}
\newtheorem{defn}[thm]{Definition}
\newtheorem{nota}[thm]{Notation}
\theoremstyle{remark}
\newtheorem{rk}[thm]{Remark}

\newtheorem{qn}[thm]{Question}

\begin{document}

\title{Divided powers on abelian varieties}
\author{Bruno Kahn}
\address{CNRS, Sorbonne Université and Université Paris Cité, IMJ-PRG\\ Case 247\\4 place
Jussieu\\75252 Paris Cedex 05\\France}
\email{bruno.kahn@imj-prg.fr}
\begin{abstract}
We prove the existence of divided powers in étale Chow groups of abelian varieties over a separably closed field, and hence of an integral lift of the Fourier transform, away from the characteristic and up to $2$-torsion. The method is to lift the Deninger-Murre Chow-Künneth projectors to integral ones, and draw consequences. Several techniques used here are new.
\end{abstract}
\date{June 5, 2026}
\maketitle

\tableofcontents

\section{Introduction}

Let $A$ be an abelian variety of dimension $g$ over a field $k$. In \cite[Rem. 4.1]{ln}, Hélène Esnault gave an example, with $g=2$, where there exists a divisor class $L\in CH^1(A)$ such that $L^2$ is not divisible by $2$ in the Chow group $CH^2(A)$ and even in its étale version $CH^2_\et(A)$, where
\[CH^i_\et(A):=H^{2i}_\et(A,\Z(i))\] 
(see \cite{cycle-etale} for étale motivic cohomology). This is significant because, for any value of $g$, the degree of $L^g$ is divisible by $g!$ \cite[\S16]{mumford}. We recall her argument in \S \ref{s2} and give a slightly stronger statement, where $CH^2_\et(A)$ is replaced by $H^4_\cont(A,\Z_2(2))$ (continuous étale cohomology).   

In the same negative direction, Philip Engel, Olivier de Gaay Fortman and Stefan Schreieder have proven in \cite[Th. 1.1]{EGS} that, for $k=\C$, $g\ge 4$ and $A$ very general, principally polarised by the divisor $\Theta$, the cohomology class of any cycle of codimension $c\in [2,g-1]$ is an even multiple of $\cl(\Theta)^c/c!$, hence a fortiori that $\Theta^c$ is not divisible by $c!$ in $CH^c(A)$. In \cite{EGS2}, the multiple $2$ is replaced by $6$ for $g=6$.

Suppose $k$ separably closed. If $g=2$, then  $CH^2(A)_{\deg=0}$ is divisible and, moreover, $CH^2(A)\allowbreak\to CH^2_\et(A)$ is an isomorphism  thanks to the exact sequence
\[0\to CH^2(A)\to CH^2_\et(A)\to H^0(A,\sH^3_\et(\Q/\Z(2)))\to 0\]
of \cite[Prop. 2.9]{cycle-etale} and the fact that $H^0(A,\sH^3_\et(\Q/\Z(2)))=0$ since $\cd(k(A))=2$. Thus Esnault's example cannot exist in this case. Similarly, the following proposition shows that the counterexamples of \cite{EGS,EGS2} disappear when we replace Chow groups by étale Chow groups. Here $g$ is arbitrary.

\begin{prop}\label{t1} Let $x\in CH^i_\et(A)$. Then $x^r$ is divisible by $(r!)_p$ in $CH^{ir}_\et(A)$ for any $r\ge 1$, where $p$ is the exponential characteristic of $k$ and $(n)_p$ denotes the largest divisor of an integer $n$ which is prime to $p$.
\end{prop}

The proof, given in \S \ref{s4}, is almost trivial in view of a result of \cite{cycle-etale}; I feel sheepish not to have noticed it at the time. (In fact, \cite[Th. 1.22 and (1.8)]{glr} is already sufficient.)

In Proposition \ref{t1}, we could say that $CH^*_\et(A)$ has ``weak'' divided powers; note that in view of torsion in this ring, they are by no means uniquely determined. 
This raises the issue of strong divided powers, i.e. the existence of functions $\gamma_r:CH^i_\et(A)\to CH^{ir}_\et(A)$ such that, identically
\begin{enumerate}
\item $\gamma_0(x)=1$, $\gamma_1(x)=x$;
\item $\gamma_n(\lambda x) = \lambda^n \gamma_n(x)$ for any $\lambda$;
\item $\gamma_n(x+y) = \sum_{r+s=n} \gamma_r(x)\gamma_s(y)$:
\item $\gamma_m(x)\gamma_n(x) = \binom{m+n}{m} \gamma_{m+n}(x)$;
\item $\gamma_n(\gamma_m(x))= \frac{(mn)!}{(m!)^n n!} \gamma_{mn}(x)$.
\end{enumerate}

It turns out that the solution to this problem is essentially positive --- except that, besides the exponential characteristic, the prime $2$ comes to haunt us: one \emph{cannot} refine Proposition \ref{t1} to obtain such divided powers, see Remark \ref{r1} below. However, the $2$-torsion subgroup ${}_2CH^*_\et(A)[1/p]$ is an ideal in the ring $CH^*_\et(A)[1/p]$; write $\oCH^*_\et(A)$ for the quotient.\footnote{Since the  torsion subgroup of $CH^*_\et(A)[1/p]$ is divisible (Proposition \ref{p4} c)), $CH^*_\et(A)[1/p]$ and $\oCH^*_\et(A)$ are isomorphic as groups but probably not as rings.} Then:

\begin{thm}\label{t5} a) There exists a canonical divided power structure on the ideal $\oCH^{>0}_\et(A)$ of $\oCH^*_\et(A)$.\\
b) If $f:A\to B$ is a homomorphism of abelian varieties, then $\gamma_n(f^*x) = f^*\gamma_n(x)$ for any $x\in  \oCH^{>0}_\et(B)$.
\end{thm}

A consequence of Theorem \ref{t5} is that the Chern character of any line bundle over $A$ admits a canonical lift to $\oCH^*_\et(A)$. (Here we use the fact that $CH^*(A)\otimes \Q\iso CH^*_\et(A)\otimes \Q$ \cite[Th. 2.6 c)]{cycle-etale}.) As a special case, we get part of the following corollary.

\begin{cor}\label{c12}  The Fourier transform of Beauville \cite{beauville} admits a canonical  lift $\sF_A$ to $\oCH^*_\et(A)$ verifying the  identities of loc. cit., Prop. 3 (i), (ii) and (iii): inversion formula, exchange of intersection product and Pontryagin product and commutation with isogenies. 
\end{cor}

In \S \ref{s11.5}, we show that several identities of Beauville in \cite{beauville} have integral lifts, and use this to give an integral lift of Scholl's formula in \cite{scholl} for the Chow-Künneth projectors of Deninger-Murre when $A$ is principally polarised, and of Suh's formula in \cite{suh} for Jacobians of curves (see below for more on the Deninger-Murre projectors).

\begin{cor}\label{c8a} The augmentation ideal of $\oCH_*^\et(A)$ provided with the Pontryagin product admits a canonical divided power structure.
\end{cor}

\begin{rk}\label{r1} Corollary \ref{c12} is optimal:  by \cite[Th. 3.11]{mo-po}, one cannot lift its Fourier transform to $CH^*_\et(A)[1/p]$ in general, already when $A$ is an elliptic curve (note that Chow groups and étale Chow groups agree in this case).
\end{rk}

We now come to Chow-Künneth projectors. Here we don't need to use $\oCH^*_\et(A)$. For convenience, here is a definition.

\begin{defn}\label{d4} Let $\sA\to S$ be an abelian scheme of relative dimension $g$, where $S$ is a connected smooth scheme over a field of exponential characteristic $p$. Orthogonal projectors with sum $1$  $(\pi^i)_{0\le i\le 2g}$ in $CH^g(\sA\times_S \sA)$ (resp. in $CH^g_\et(\sA\times_S \sA)[1/p]$) are called \emph{integral} (resp. \emph{étale integral}) Deninger-Murre projectors (DM projectors, for short) if
\begin{itemize}
\item For any $i$ and any prime $l\ne p$, $\cl_l(\pi^i)$ is the $i$-th Künneth projector of $H^*(\sA_{\bar \eta},\Q_l)$, where $\sA_{\bar \eta}$ is the geometric generic fibre of $\sA$; here $\cl_l$ is the $l$-adic cycle class map (resp. the étale $l$-adic cycle class map  of  \cite[\S 3A]{cycle-etale}).
\item For any $i$ and any integer $n$, $n_\sA^*\circ \pi^i=n^i\pi^i$, where $n_A$ is multiplication by $n$ on $\sA$. 
\end{itemize}
\end{defn}

\begin{thm}\label{t3} a) If $A$ is an abelian $k$-variety with $k$ separably closed, the rational DM projectors of \cite{dm} lift to a system of étale integral DM projectors $(\pi^i_A)$. This set is unique up to conjugation by an element of the form $1+x$ with $2x=0$.\\
b) If $B$ is another abelian $k$-variety, then $2\pi^i_A\circ f^* = 2f^*\circ \pi^i_B$ in $CH^{g}_\et(A\times B)[1/p]$ for any homomorphism $f:A\to B$.  Here $g=\dim A$.\\
c) If $A$ is principally polarised in a), we may choose $(\pi^i_A)$ \emph{self-conjugate}, i.e. ${}^t\pi^i_A = \pi^{2g-i}_A$.
\end{thm}

(In b), the factor $2$ is rather unsubstantial: among the correspondences $\pi^j_A\circ f^*\circ \pi^i_B$ for $i\ne j$, the only possibly nonzero ones are for $j=i-1$, and those are $2$-torsion.)

The proof of Theorem \ref{t3} is in two steps. First a weaker statement up to multiplication by $2$,  Theorem \ref{t4}. Then we get rid of this factor $2$ by a deformation argument which rests on the (non trivial!) special case of elliptic curves: Theorem \ref{t6} and Proposition \ref{pDM}. The reader may consult Theorem \ref{t7} for a more complete and cleaner statement in $\oCH^*_\et(A)$.

The following corollary was my initial motivation for this work.

\begin{cor}\label{c4} Let $\Ab$ be the category of abelian $k$-varieties, with morphisms the homomorphisms of abelian varieties. Let $\sM$ be the category of effective Chow motives with integral coefficients \cite{scholl}. Then the additive functor of \cite[Cor. 5.2]{scholl}
\[h^1\otimes \Q:\Ab\otimes \Q\to \sM\otimes \Q\] 
induced by the first Chow-Künneth projectors of \cite{dm} lifts to an additive, fully faithful functor 
\[h^1_\et:\Ab[1/p]\to \sM_\et,\] 
where $\sM_\et$ is the category of effective étale Chow motives (see \S \ref{s1}).
\end{cor}

I doubt that this functor lifts to a functor $\Ab\to \sM$, although I don't have a counterexample.

\begin{rk} The proof of the full faithfulness of $h^1_\et$ rests on an adjunction statement closely related to the functors $\LAlb$ and $\RPic$ of \cite{bvk}: Theorem \ref{tA1}. This recovers the full faithfulness of $h^1\otimes \Q$ \cite[Cor. 5.10]{scholl}\footnote{In the proof of loc. cit., the 4th line of the computation should be dropped.} with a completely different proof, which does not involve a hard Lefschetz isomorphism theorem.
\end{rk}

While in \cite{dm} the existence of the Chow-Künneth projectors is deduced from that of the Fourier transform, here we proceed in the opposite way, proving first Theorem \ref{t3}, deducing Theorem \ref{t5} and finally Corollary \ref{c12}. Contrary to my initial expectation, Proposition \ref{t1} is not used.

\subsection*{Abelian schemes}  In \S \ref{s9}, we extend the above story to abelian schemes $\sA$ over a smooth base $S$, up to losing some torsion. In the first version of this paper, I had claimed that ``all the above results hold'' after inverting a suitable integer. Not only was the proof totally insufficient, but such a claim cannot be true since divided powers have no reason to exist in general, e.g. if $\sA=S$. The correct version is this:

\enlargethispage*{20pt}

\begin{thm}\label{c1} Let $S$ be a smooth scheme over a field $k$, and let $\sA$ be an abelian $S$-scheme of relative dimension $g$. Let
\[M= \inf(2g,\cd(S)+2).\] 
Then\\
a) There is a unique set $(\pi^i_\sA)_{0\le i\le 2g}$ of DM projectors in $CH^g_\et(\sA\times_S \sA)[1/M!p]$. To just get Chow-Künneth projectors (without requiring the DM condition), it suffices to invert $N!p$ where
\[N=  \inf(2g,\cd(S)).\]
b) Let $\xi\in CH^1_\et(\sA)$ be the class of a relatively symmetric line bundle (i.e. $(-1)_\sA^*\xi = \xi$). Then, for any $i\ge 1$, there exists a unique element $\gamma_i(\xi)\in \pi^{2i}_\sA CH^i_\et(\sA)[1/M!p]$ such that $i!\gamma_i(\xi)=\xi^i$. The divided power identities (1) -- (4) hold.\\
c) In $CH^g_\et(\sA\times_S \hat{\sA})[1/M!p]$, there exists an integral version of the Deninger-Murre Fourier transform enjoying the properties of \cite[Cor. 2.22]{dm}.
\end{thm}

Here $\cd(S)$ is the étale cohomological dimension of $S$; recall that $\cd(S) \le \cd(k)+2\dim S$, and even $\cd(S) \le \cd(k)+\dim S$ if $S$ is affine. 

A similar bound was obtained  in \cite[Th. 5.5]{moonenetal} for ordinary Chow groups, where $M$ is replaced by $2g+\dim S+1$ (there are more integral results in §4 of this paper). The proofs are completely different: those of \cite{moonenetal} rest essentially on Pappas' integral refinement of the Riemann-Roch theorem \cite{pappas}, while the present proofs rest on the excellent control one has on the torsion subgroup of $CH^*_\et(\sA)[1/p]$, see \S\S \ref{s1}, \ref{s3} and \ref{s12.1}.

\subsection*{Algebraic and étale  cohomology classes} We come back to the case of a separably closed base field $k$. Call an $l$-adic cohomology class of a smooth $k$-variety $X$ \emph{algebraic} (resp. \emph{étale}) if it is in the image of the $l$-adic cycle class map (resp. of the étale $l$-adic cycle class map). If $k\subseteq \C$, both cycle class maps admit refinements to maps to Betti cohomology (in the étale case, thanks to \cite[Déf. 2.4]{ayoub-betti}), and we can talk of algebraic and étale  classes in Betti cohomology. The subgroups of algebraic and étale  classes coincide rationally and the latter is pure in $H^*(X)$ (see Proposition \ref{p4a} a) below); therefore, when $X$ is projective, it is contained in the subgroup of Hodge or Tate classes. This readily shows that the (rational) Hodge or Tate conjectures are equivalent to their integral versions involving étale  classes (see \cite{ro-sri} for developments of this remark). When the latter conjectures are not known, étale  classes thus allow us to reformulate the issue of an ``integral Hodge or Tate conjecture'' in the following unconditional form:

\enlargethispage*{20pt}

\begin{qn} When are étale  classes algebraic?
\end{qn}

In this spirit, the results of \cite{EGS,EGS2} mentioned at the beginning of this introduction show that, in general, algebraic classes are not stable under divided powers. By contrast, the previous results show:

\begin{cor}\label{c3} For abelian varieties, étale  classes are stable under divided powers and the Fourier transform. The Künneth projectors are étale.  The groups of étale  classes of $A$ and its dual $\hat{A}$ are isomorphic.
\qed
\end{cor}
 
 To conclude, these results confirm that étale Chow groups are a receptacle in which (basically) ``everything goes well'', in contrast to ordinary Chow groups: étalification in some sense removes the asperities of torsion in the latter.\footnote{For example, the ``étale Griffiths group'' is divisible by Proposition \ref{p4a} a), contrary to the classical one in characteristic $0$ \cite{bl-es,schoen}.} These asperities, on the other hand, are e.g. sources of counterexamples to rationality, as for example in \cite[Cor. 1.4]{EGS}. My viewpoint is that both theories are relevant and one should play with them according to needs. See \S\S \ref{sB.6} and \ref{sB.7} for possible interactions of the present ideas with ordinary Chow groups.

\subsection*{Acknowledgement} I thank Olivier Benoist for pointing out the preprint \cite{EGS2}, and Alexei Skorobogatov for a helpful discussion.

\section{Étale motivic cohomology of abelian varieties}\label{s1}

Let $k$ be a field of exponential characteristic $p$, and take a prime number $l$, possibly equal to $p$. Recall

\begin{prop}[{\cite[Cor. 3.5]{cycle-etale} and \cite[Prop. 1]{indec}}]\label{p4a} Let $X$ be a smooth $k$-variety. Let $(i,j)\in \Z$.\\
a)  The étale cycle class map of  \cite[\S 3A]{cycle-etale}
\[\cl_l:H^j_\et(X,\Z(i))\otimes \Z_l\to H^j_\cont(X,\Z_l(i))\]
has divisible kernel and torsion-free cokernel. Here  the target is Jannsen's continuous étale cohomology \cite{jannsen} if $l\ne p$ and logarithmic Hodge-Witt cohomology if $l=p$.\\
b) Suppose $k$ algebraically closed and $X$  projective. For  $j\ne 2i$, $H^j_\et(X,\Z(i))\otimes \Z_{(l)}$ is a direct sum of a group of finite exponent and a divisible group; in particular, $H^j_\et(X,\Z(i))\otimes \Q_l/\Z_l=0$ and the image of $\cl_l$ is finite if $l\ne p$.
\end{prop}

\begin{defn}\label{d2} If $M$ is an abelian group, we say that an element $\alpha\in M\otimes \Q$ is \emph{integral} (with respect to $M$) if it comes from some element of $M$.
\end{defn}

\begin{rk}\label{r2} In Definition \ref{d2},\\
a) $\alpha$ is integral with respect to $M$ if and only if it is integral with respect to $M/M_\tors$, and its lift is unique in the latter case.\\
b) To be integral is the same as to vanish in $M\otimes \Q_l/\Z_l$ for all $l$. This shows that $\alpha$ is integral with respect to $M$ if and only if $\alpha\otimes \Q_l$ is integral with respect to $M\otimes \Z_l$ for all $l$.
\end{rk}

\begin{cor}[Integrality lemma]\label{c8} Let $\alpha\in H^j_\et(X,\Q(i))$, and let $N>0$. Suppose that $\cl_l(\alpha)\in H^j_\cont(X,\Q_l(i))$ is integral with respect to $H^j_\cont(X,\Z_l(i))$ for all $l\nmid N$. Then $\alpha$ is integral with respect to $H^j_\et(X,\Z(i))[1/N]$.
\end{cor}

\begin{proof}  In view of remark \ref{r2} b), this follows from Proposition \ref{p4a} a) via an easy diagram chase.
\end{proof}

\begin{prop}[{\cite[Th. 15.1]{milne}}]\label{p5} Let $A$ be an abelian $k$-variety of dimension $g$, with $k$ separably closed. If $l\ne p$, $H^*_\cont(A,\Z_l)$ is an exterior algebra on $H^1_\cont(A,\Z_l)$; in particular $H^*_\cont(A,\Z_l)$ is torsion-free.
\end{prop}

The same holds for Betti cohomology when $k=\C$ \cite[\S 1, (4)]{mumford}.

\begin{prop}\label{p4} In Proposition \ref{p4a}, suppose that $X=A$ is an abelian variety. Suppose also $k$ separably closed and $l\ne p$. Then\\
c) There is a canonical isomorphism
\[H^{j-1}_\cont(A,\Z_l)\otimes \Q_l/\Z_l(i)\iso H^j_\et(A,\Z(i))\{l\}.\]
If $j\ne 2i$, $H^j_\et(A,\Z(i))[1/p]$ is divisible.\\
d) If $(i,j),(i',j')$ are two pairs, the cup-product
\[ H^j_\et(A,\Z(i))\times H^{j'}_\et(A,\Z(i'))\to H^{j+j'}_\et(A,\Z(i+i'))\]
restricts to $0$ on torsion subgroups. \\
e) For $n\in \Z$, the contravariant action of $n_A\in \End(A)$ on $H^j_\et(A,\Z(i))_\tors$ is multiplication by $n^{j-1}$ and its covariant action is multiplication by $n^{2g-j+1}$, at least away from $p$.\\
f) If $j\ne 2i$, $H^j_\et(A,\Z(i))[1/p]$ is divisible.
\end{prop}

\begin{proof} In c), the isomorphism is the composition of two  isomorphisms
\[H^{j-1}_\cont(A,\Z_l)(i)\otimes \Q_l/\Z_l\iso H^{j-1}_\et(A,\Q_l/\Z_l(i)) \iso H^j_\et(A,\Z(i))\{l\} \]
where the second one follows from Proposition \ref{p4a} b) and the first holds by Proposition \ref{p5}. 
By Proposition \ref{p4a} b), if $j\ne 2i$ then the image of the étale cycle class map is finite, hence $0$ by Proposition \ref{p5}, hence $H^j_\et(A,\Z(i))[1/p]$ is divisible by Proposition \ref{p4a} a). 
d) follows from c), since the tensor product of two torsion divisible groups is $0$, and e) also follows from c). Finally, f) follows from Propositions \ref{p4a} and \ref{p5}.
\end{proof}

\subsection*{Correspondences and motives} We shall use three categories of (pure, effective) motives: $\sM$ (Chow motives), $\sM_\et$ (étale Chow motives) and $\hat{\sM}$ (cohomological motives).  There are $\otimes$-functors
\begin{equation}\label{eq29}
\sM\by{\alpha^*} \sM_\et \by{\hat{R}} \hat{\sM}.
\end{equation}
 
 If $X$ is a smooth projective $k$-variety, we write $h(X)$, $h_\et(X)$, $\hat{h}(X)$ for its motive in $\sM,\sM_\et,\hat{\sM}$ respectively. If $Y$ is another, connected, one, morphisms from its motive to that of $X$ are elements of the groups of correspondences
\[CH^{\dim Y}(Y\times X),\; CH^{\dim Y}_\et(Y\times X)[1/p],\; \hat{H}^{2\dim Y}(Y\times X,\dim Y)\]
where
\[\hat{H}^j(Z,i):= \prod_{l\ne p}H^j_\cont(Z,\Z_l(i)).\]

This extends to nonconnected $Y$ by additivity as usual. The functor $\hat{R}$ is induced by the étale cycle class maps $\cl_l$, and $\alpha^*\otimes \Q$ is an equivalence of categories \cite[Th. 2.6 c)]{cycle-etale}. Correspondences $\gamma,\delta$ from $Z$ to $Y$ and from $Y$ to $X$ compose according to the usual rule
\[\delta\circ \gamma = (p_{13})_*(p_{12}^*\delta \cdot p_{23}^*\gamma)\]
where $p_{ij}$ are the projections of $X\times Y\times Z$ on the two-fold factors. To justify that composition is well-defined and associative in $\sM_\et$, the simplest is to use the formalism of the six operations for étale motives, as in \cite{ayoub} or \cite{cis-deg}.\label{page} \footnote{\label{f2} If one wishes not to invert $p$ in characteristic $p$, the six operations are not available anymore because $\A^1$-invariance is lost. See \cite[App. A]{cell0} for a programme to define such a refined category; to the best of my knowledge, the required properties have still not been written up.} 

If $k= \C$, the adic realisation $\hat{R}$ factors through the Betti realisation
\[\sM_\et\by{R_B}\sM_B\]
where $\sM_B$ is the category of cohomological motives with respect to Betti cohomology $H_B$, and $H^j_B(Z,\Z(i))\otimes\prod_l\Z_l\iso \hat{H}^j(Z,i)$ by Artin's comparison theorem. The reader only interested in this case can replace $\hat{H}$ by $H_B$ in the sequel.

We shall also use the notation
\[\sM^\ab\by{\alpha^*} \sM_\et^\ab \by{\hat{R}} \hat{\sM}^\ab\]
for the thick (= full, stable under direct summands) subcategories generated by motives of abelian varieties; similarly for $\sM_B^\ab$. We then have

\begin{cor}\label{c2} Composition in $\sM_\et^\ab$ restricts to $0$ on torsion correspondences.
\end{cor}

\begin{proof} This follows from Proposition \ref{p4} d).
\end{proof}

\section{Proof of Proposition \ref{t1}}\label{s4}

By Proposition \ref{p5}, the  even part of the cohomology algebra $H^*_\cont(A,\Z_l)$ has a divided power structure \cite{papy,revoy}. Thus $\cl_l(x^r)=\cl_l(x)^r$ is divisible by $r!$ and we conclude with Corollary \ref{c8}, applied to $\alpha=\frac{x^r}{r!}$. \qed

\begin{rk} \label{r0} If $x\in H^j_\et(A,\Z(i))$ with $j\ne 2i$, then $x^r$ is divisible by \emph{any} positive integer in $H^{jr}_\et(A,\Z(ir))[1/p]$ by Proposition \ref{p4} f). If $j$ is odd and $p\ne 2$, we even have $x^2=0$ because $2x^2=0$ \emph{a priori} \cite[Th. 15.9]{mvw}.
\end{rk}

\section{Esnault's example, refined}\label{s2}

Here we take $g=2$. Further, we assume that $A$ is the Jacobian of a smooth projective curve $C$ having a rational point $c$. This provides an embedding
\[i:C\inj A\]
sending $c$ to $0$ and realising $i(C)$ as a theta divisor. 

\begin{prop}\label{p0} Let $K_C\in \Pic(C)$ be the class of the canonical divisor on $C$, and let $x$ be the class of $i(C)$ in $\Pic(A)$. Then $\deg(x^2)=2$ and the Albanese class of $x^2-2[0]\in CH^2(A)_0$ in $A(k)$ is $K_C-2[c]$. In particular, $x^2$ is divisible by $2$ in $CH^2(A)$ or $CH^2_\et(A)$ if and only if $K_C$ is divisible by $2$ in $\Pic(A)$. If $p\ne 2$, the same holds for $\cl_2(x)$.
\end{prop}

\begin{proof}
By \cite[Ch. II, Prop. 8.20]{hartshorne}, we have
\[K_C= i^*(K_A+[i(C)])=i^*i_*[C]\]
using that $K_A=0$ \cite[\S 4, (ii)]{mumford}. Therefore
\[ x^2=i_*i^*i_*[C] =i_*K_C.\]

This already gives $\deg(x^2) = 2g-2 = 2$, and moreover
\[x^2-2[0] = i_*(K_C-2[c])\]
hence the conclusion (here we use that the composition $A(k)=\Pic^0(A)\allowbreak\by{i_*}CH_0(A)_0\by{a} A(k)$ is the identity, where $a$ is the Albanese map.) To deal with $\cl_2(x)$, we reason as in the proof of Proposition \ref{t1}.
\end{proof}

One example in \cite[\S 4]{ln} (attributed to Serre) where the condition of Proposition \ref{p0} fails is $k=\C(t)$, $C$ given by the affine equation $y^2=x^6-x-t$.
\bigskip

\emph{From now on and until \S \ref{s9}, the field $k$ is supposed to be separably closed.}

\section{The torsion bimodule}\label{s3}

Let $X$ be a smooth projective $k$-variety. From now on, we abbreviate $H^{j}_\cont(X,\Z_l(i))$ to $H^j_l(X,i)$ and $H^j_l(X,0)$ to $H^j_l(X)$ if  $l\ne p$. 
Consider the following two actions:
\begin{itemize}
\item the étale intersection product 
\[H^{j-1}_\et(X,\Q_l/\Z_l(i))\times H^{j'}_\et(X,\Z(i'))\to H^{j+j'-1}_\et(X,\Q_l/\Z_l(i+i')),\] 
denoted by $\cdot$;
\item the $l$-adic product 
\[H^{j-1}_\et(X,\Q_l/\Z_l(i))\times H^{j'}_l(X,i')\to H^{j+j'-1}_\et(X,\Q_l/\Z_l(i+i')),\] 
denoted by $\cup$.
\end{itemize}

\begin{lemma}\label{l4} For $(x,y)\in H^{j-1}_\et(X,\Q_l/\Z_l(i))\times H^{j'}_\et(X,\Z(i'))$, we have
\[x\cdot y =\cl_l(x)\cup y.\]
\end{lemma}

\begin{proof}  This follows from the fact that, in \cite{cycle-etale}, the morphism (3-2) (defining $\cl_l$) is compatible with products.
\end{proof}

Given an additive category $\sA$, recall that an \emph{$\sA$-bimodule} is an additive bifunctor on $\sA^\op\times \sA$. Let 
\[(\Q/\Z)'=\bigoplus_{l\ne p} \Q_l/\Z_l.\]

 For $r\in \Z$, the two actions described above provide the assignment  $(Y,X)\allowbreak\mapsto H^{2\dim Y-r}_\et(Y\times X, (\Q/\Z)'(\dim Y))$ with structures of bimodule over $\sM_\et$ and $\hat{\sM}$ respectively, and Lemma \ref{l4} implies

\begin{lemma}\label{l5} These two bimodule structures are compatible via $\hat{R}$.\qed
\end{lemma}

We denote this common bimodule by $t(r)$. 
For abelian varieties $A,B$ with $\dim A=g$, the proof of Proposition \ref{p4} c) interprets $t(1)$ both as $H^{2g}_\et(A\times B, \Z(g))[1/p]_\tors$ and as $\hat{H}^{2g-1}(A\times B, g)\otimes \Q/\Z$. In the rest of this section, we compute the  $\hat{\sM}^\ab$-bimodule structure of $t(r)$. For this, we note that as a consequence of Proposition \ref{p5}, the Künneth formula for $H^*_l(A\times B)$ holds integrally for all $l\ne p$. Moreover, Poincaré duality also holds integrally for $A$ thanks to \cite[§11, Prop. 7]{bbki} and the fact that the trace map $H_l^{2g}(A,g)\to \Z_l$ is an isomorphism. We therefore have a Künneth decomposition in $\hat{\sM}$:
\begin{equation}\label{eq8}
\hat{h}(A) = \bigoplus_{i=0}^{2g} \hat{h}^i(A)
\end{equation}
corresponding to the Künneth projectors $\hat{\pi}^i_A$ in
\[\hat{H}^{2g}(A\times A,g)\simeq \prod_{i=0}^{2g} \End_{\hat{\Z}}(\hat{H}^i(A)).\]

We first recall the structure of the tautological $\hat{\sM}^\ab$-bimodule $t(0)=\Hom_{\hat{\sM}^\ab}$: in view of \eqref{eq8}, it may be written $\bigoplus_{i,j}\hat{\sM}(\hat{h}^i(A),\hat{h}^j(B))$, where
\[\hat{\sM}(\hat{h}^i(A),\hat{h}^j(B))=
\begin{cases}
0&\text{if } i\ne j\\
\Hom_{\hat{\Z}}(\hat{H}^i(A),\hat{H}^i(B))& \text{if } i= j.
\end{cases}
\]

In particular, $\hat{\pi}^i_B\circ f= f\circ \hat{\pi}^i_A$ for any $f\in \hat{\sM}(\hat{h}(A),\hat{h}(B))$, and the $\hat{\pi}^i_A$ are central in $\hat{H}^{2g}(A\times A,g)$.

\begin{lemma}\label{l6} Let $T(r)$ be the $\hat{\sM}^\ab$-bimodule induced by $(A,B)\mapsto \hat{H}^{2g-r}(A\times B,g)$. Then
\[T(r)(\hat{h}^i(A),\hat{h}^j(B))=
\begin{cases}
0&\text{if } j\ne i-r\\
\Hom_{\hat{\Z}}(\hat{H}^i(A),\hat{H}^{i-r}(B))& \text{if } j= i-r.
\end{cases}
\]
In particular, $\hat{\pi}^{i-r}(B)\bullet x= x\bullet \hat{\pi}^{i}(A)$ for any $x\in T(r)(\hat{h}(A),\hat{h}(B))$, where $\bullet$ denotes the bimodule action.\\
The same holds, mutatis mutandis, for $t(r) = T(r)\otimes \Q/\Z$.
\end{lemma}

\begin{proof} Using again the Künneth formula and Poincaré duality, we get an isomorphism
\begin{equation}\label{eq20}
\hat{H}^{2g-r}(A\times B,g)\simeq \bigoplus_{i= 0}^{2g} \Hom_{\hat{\Z}}(\hat{H}^i(A),\hat{H}^{i-r}(B))
\end{equation}
and the lemma follows by bookkeeping.
\end{proof}

For later use, we consider a slightly more general situation, replacing $B$ by an arbitrary connected smooth projective variety $X$. By Proposition \ref{p5}, the Künneth formula for $\hat{H}^*(A\times X)$ still holds integrally. This, plus Poincaré duality for $A$, still yields the isomorphism \eqref{eq20} for any $r\in \Z$. Torsion in the cohomology of $X$ makes it difficult to talk of its integral Künneth projectors, but those of $A$ still act on the left hand side, yielding the decomposition on the right hand side. For $r=1$ and $l\ne p$, we have a short exact sequence
\begin{multline*}
0\to H^{2g-1}_l(A\times X,g)\otimes \Q_l/\Z_l\to H^{2g-1}_\et(A\times X,\Q_l/\Z_l(g))\\
\to H^{2g}_l(A\times X,g)\{l\}\to 0
\end{multline*}
from which the action of $\pi^1_l$ (the $l$-component of $\hat{\pi}^1$) cuts off the exact sequence
\begin{multline*}
0\to  \Hom_{\Z_l}(H^1_l(A),H^0_l(X))\otimes \Q_l/\Z_l\to  H^{2g-1}_\et(A\times X,\Q_l/\Z_l(g))\circ \pi^1_l\\
\to \Hom_{\Z_l}(H^1_l(A),H^1_l(X))\{l\}\to 0.
\end{multline*}

But $H^1_l(X)$ is always torsion-free and $H^0_l(X)=\Z_l$, so this exact sequence boils down to isomorphisms
\begin{equation}\label{eq21a}
A(k)\{l\}\simeq H^1_l(A)^*\otimes \Q_l/\Z_l\iso  H^{2g-1}_\et(A\times X,\Q_l/\Z_l(g))\circ \pi^1_l.
\end{equation}

\section{$1$-cocycles} \label{s6}

\subsection{Hochschild cohomology} Let $\Lambda=\Z[\Z-\{0\}]$ be the group algebra of the multiplicative monoid $\Z-\{0\}$; if $M$ is a $\Lambda$-bimodule, we have as usual its Hochschild cohomology
\[HH^r(\Lambda,M) =\Ext^r_{\Lambda\otimes \Lambda}(\Lambda,M)\]
where $\Lambda$ is considered as a $\Lambda$-bimodule via left and right action. In particular,
\[HH^1(\Lambda,M) = \frac{\{f:\Z-\{0\}\to M\mid f(mn) = m\diamond f(n) + f(m) \diamond n\}}{\{ m\mapsto m\diamond a - a \diamond m\}} \]
where $\diamond$ denotes the left and right actions. Since $\Z-\{0\}$ is commutative, the $1$-cocycle relation implies in particular the identity in $(m,n)$
\begin{equation}\label{eq6a}
m\diamond f(n) - f(n) \diamond m=n\diamond f(m) - f(m) \diamond n.
\end{equation}

\subsection{Some numerology} 

\begin{defn}\label{eq16}
For any $i>j\ge 0$, we set
\[
w_{i,j} = \gcd\nolimits_{m\ne 0}(m^i-m^{j})
\]
and let $w_{i,j}:=1$ if $i\ge 0$ and $j<0$. 
\end{defn}

A stable version appears in \cite[2.8]{soule}: Soulé's $w_{r}$ is the supremum of our $w_{i,j}$ for $i-j=r$. 

Here are some elementary properties of $w_{i,j}$:

\begin{lemma}\label{l11} a) Let $l$ be a prime number and let $s>0$. If $l$ is odd, then $l^s\mid w_{i,j}$ $\Rightarrow$ $l^s-l^{s-1}\mid i-j$. For $l=2$,  $4\mid w_{i,j}$ $\Rightarrow$ $i-j$ is even and $2^s\mid w_{i,j}$ $\Rightarrow$ $2^{s-2}\mid i-j$ for  $s>2$. In particular,  $w_{i,j}=2$ if $i-j$ is odd. \\
b) If $j=0$, then $w_{i,j}=1$.\\
c) The converse to a) holds provided $j\ge s$.
\end{lemma}

\begin{proof} a) Suppose $l$ odd. Then $(\Z/l^s)^*$ is cyclic of order $l^s-l^{s-1}$. Taking an integer $m$ whose class generates this group gives a) in this case. For $l=2$, $(\Z/2^s)^*=\{\pm 1\} \times (1+4 \Z/2^s)\simeq \Z/2\times  \Z/2^{s-2}$ when $s>1$; taking either $m=-1$ when $s=2$ or $m=5$ when $s>2$ gives a) in that case. 
If $j=0$, then $l\nmid l^i-1$ proves b). Conversely, to get c) we observe that, under the hypothesis, if $l\nmid m$ then  $l^s\mid m^{i-j}-1$ and if $l\mid m$ then $l^s\mid m^j$.
\end{proof}

\subsection{Cocycles and coboundaries}\label{s6.3}
For $i> j$, define
\[R_{i,j} =\Lambda\otimes\Lambda/\langle ([m]-m^{j})\otimes 1, 1\otimes ([n]-n^i)\rangle.\]

A $\Lambda$-bimodule $M$ is \emph{of level $(i,j)$} if the action of $\Lambda\otimes \Lambda$ factors through $R_{i;j}$. Here the $1$-cocycle relation is
\begin{equation}\label{eq7}
 f(mn) =m^{j} f(n)+ n^if(m)
\end{equation} 
and \eqref{eq6a} becomes
\begin{equation}\label{eq6}
(m^i-m^j) f(n) =(n^i-n^j) f(m).
\end{equation}

For convenience, set
\[\theta(n) = \frac{n^i-n^{j}}{w_{i,j}}.\]

Then $\theta$ is a $1$-cocycle with values in $\Z$ provided with its obvious level $(i,j)$-structure, and so is $n\mapsto \theta(n)b$ for any $b\in M$; we call such $1$-cocycles \emph{refined coboundaries}.

\begin{prop}\label{l7} a) Let $M$ be a $\Lambda$-bimodule of level $(i,j)$. Then $HH^0(\Lambda,M)\allowbreak = {}_{w_{i,j}} M$ and $w_{i,j} HH^1(\Lambda,M) =0$. In particular, $HH^0(\Lambda,M)=HH^1(\Lambda,M)\allowbreak =0$ if $j= 0$.\\
b) Suppose that $j=i-1>0$ and $w_{i,j} M=2M=0$. Then $f(n)=0$ if $n$ is a square or if $4\mid n$. Moreover, $f(2n)=f(2)$ if $n$ is odd.
\end{prop}

\begin{proof} a) The first claim is obvious.  Let now $f$ be a $1$-cocycle. Take $n_1,\dots, n_r$ such that $\gcd_s(n_s^i-n_s^j)=w_{i,j}$; choose $a_1,\dots a_r$ such that $\sum_s a_s (n_s^i-n_s^j)=w_{i,j}$, and let $b=\sum_s a_s f(n_s)$. Then, for all $n$,
\[w_{i,j}f(n) = \sum_s a_s (n_s^i-n_s^j)f(n)=\sum_s a_s (n^i-n^j)f(n_s)= (n^i-n^j) b\]
thanks to \eqref{eq6}, hence the second claim. The last one follows from Lemma \ref{l11} b).

b) In this case, \eqref{eq7} may equally be written
\[ f(mn) =m f(n)+ nf(m).\]

For $n=m$, we get the first claim;  then $f(4m)=m^j f(4) + 4^j f(m)=0$. The last identity follows similarly.
\end{proof}

\section{A weak form of Theorem \ref{t3}}\label{s7}

\subsection{A weak version of Theorem \ref{t3} a) (I)}\label{s6.1}

\begin{thm}\label{t4} a) There exists a system $(\pi^i_A)$ of orthogonal projectors of sum $1$ in $\End_{\sM_\et}(h(A))$ such that $\pi^0_A$ is given by the origin of $A$ and, for any $n\in \Z$:
\begin{itemize}
\item $n_A^*\circ \pi^1_A = n\pi^1_A$;
\item $2n_A^*\circ \pi^i_A  =2n^i\pi^i_A$ for $i>1$.
\end{itemize}
Moreover, $n_A^*\circ \pi^i_A = n^i\pi^i_A$ for all $i$ if $n$ is a square or if $4\mid n$.\\
b) Any other such system is conjugate to $(\pi^i_A)$ by an element of the form $1+z$ with $4z=0$ and $z\circ \pi^1_A=0$. Conversely, any conjugate of $(\pi^i_A)$ by such an element verifies a).\\
c) If $x\in \End_{\sM_\et}(h(A))_\tors$ is such that $(1+x)\circ \pi^i_A = \pi^i_A\circ (1+x)$ for all $i$, then $x=0$. (Hence the element $z$ of b) is unique.)
\end{thm}

We proceed in three steps. 

\subsubsection*{Step 1} Using the adic Künneth projectors $\hat{\pi}_A$ of \S \ref{s3}, Corollary \ref{c8} shows that the Chow-Künneth projectors of \cite{dm} lift from 
$CH^g(A\times A)\otimes \Q\iso CH^g_\et(A\times A)\otimes \Q$ to  a (unique) set of orthogonal projectors with sum $1$ in $CH^g_\et(A\times A)[1/p]/\tors$.

\subsubsection*{Step 2} Using Corollary \ref{c2} and  \cite[Lemma 5.4]{jannsen2}, we may lift the projectors of Step 1 to  a set of orthogonal projectors with sum $1$ $(\pi^i_0)_{i=0}^{2g}$ in $CH^g_\et(A\times A)[1/p]$.

\subsubsection*{Step 3} Now we modify the projectors of Step 2 to achieve the conditions of Theorem \ref{t4}. We choose $\pi^0_A$ as said: it certainly verifies $m_A^*\pi^0_A=\pi^0_A$. For $i>0$, consider the $\Lambda$-bimodule structure on $CH^g_\et(A\times A)[1/p]$ given by $m\diamond x\diamond n = n^i m_A^*\circ x$. By \cite[Th. 3.1]{dm}, the $1$-coboundary
\[y_i(n)=n_A^*\circ \pi^i_0  -n^i \pi^i_0\]
takes values in $CH^g_\et(A\times A)[1/p]_\tors$, hence defines a $1$-cocycle with values in this bimodule. But
\begin{equation}\label{eq4}
y_i(n)  = y_i(n)\circ \pi^i_0.
\end{equation}

Therefore, by Lemmas \ref{l5} and \ref{l6}, 
\[y_i(n)\in t_i(1) := t(1)(h^i(A),h^{i-1}(B))\]
and the bimodule $t_i(1)$ is of level $(i,i-1)$ by Proposition \ref{p4} e). Applying Proposition \ref{l7} a), we may write 
\[y_i(n) = 
\begin{cases}
(n-1)x_1 &\text{ if } i=1\\
\frac{n^i-n^{i-1}}{2}x'_i + g_i(n)  &\text{ if } i>1
\end{cases}
\]
for some $x_1,x'_i\in t_i$, with $2g_i(n)=0$. We set $g_1(n)=0$.

For $i>1$, using the divisibility of $t_i(1)$ let $x_i\in t_i(1)$ be such that $2x_i = x'_i$, and let $x=\sum_{i>0} x_i$. Define
\begin{multline}\label{eq9}
\pi^i = (1+x) \circ \pi^i_0\circ (1+x)^{-1} = \pi^i_0+x\circ  \pi^i_0- \pi^i_0\circ x\\
=  \pi^i_0+x\circ  \pi^i_0- x\circ  \pi^{i+1}_0 =  \pi^i_0+x_i-x_{i+1}
\end{multline}
where we used again Corollary \ref{c2} for the second equality and Lemmas \ref{l5} and \ref{l6} for the third.  Then, for $n\in \Z$,
\[n_A^*\circ \pi^i - n^i\pi^i = y_i(n) + (n^{i-1} -n^i)x_i = g_i(n)
\]
which proves a). The last claim follows from Proposition \ref{l7} b) applied to $g_i(n)$.

b) Still by \cite[Lemma 5.4]{jannsen2}, any other lift of the projectors of Step 2 is of the form $\tilde \pi^i=(1+z)\pi^i(1+z)^{-1}$ with $z\in CH^g_\et(A\times A)[1/p]_\tors$. The same computation as \eqref{eq9} yields
\[\tilde \pi^i=\pi^i +z_i-z_{i+1}\]
with $z_i=z\circ \pi^i=\pi^{i-1}\circ z$. 

Suppose that the $\tilde \pi^i$  also satisfy a). For $n\ne 0$, let $g_i(n)=(n^*_A-n^i)\circ \pi^i$ as in the proof of a), and let similarly $\tilde g_i(n)=(n^*_A-n^i)\circ \tilde \pi^i$. 
Then
\begin{gather*}
(n^*_A-n^i)\circ (z_i-z_{i+1}) = (n^*_A-n^i)\circ (\tilde \pi^i-\pi^i) = \tilde g_i(n) -g_i(n) \text{ if } i>1;\\
(n^*_A-n)\circ (z_1-z_{2}) =  0.
\end{gather*}

But $(n^*_A-n^i)\circ (z_i-z_{i+1}) = (n^{i-1} -n^i)z_i$ as in the proof of a). 
Since $2( \tilde g_i(n) -g_i(n))=0$ we get $4z_i=0$ for $i>1$, and $z_1=0$. Conversely, under these conditions, the same computation shows that the $\tilde \pi^i$  also satisfy a). 

In c), the condition is equivalent to $x\circ \pi^i_A = \pi^i_A \circ x$; but $x\circ \pi^i_A = \pi^{i-1}_A \circ x$ by Lemma \ref{l6}, hence 
\[\pi^i_A\circ x= \pi^i_A\circ \pi^i_A\circ x= \pi^i_A\circ x\circ \pi^i_A = \pi^i_A\circ \pi^{i-1}_A \circ x=0\]
for all $i$, and therefore $x=0$. This concludes the proof of Theorem \ref{t3} a).

\subsection{Another lifting method} Later we shall need the following proposition:

\begin{prop}\label{l15} Let $(\bar \pi^i_A)_{i=0}^{2g}$ be the set of projectors of Step 1 in the proof of Theorem \ref{t4}, and let $(\pi^i)_{i=0}^{2g}$ be a set of lifts of the $\bar \pi^i_A$, with $\sum \pi^i=1$. Then the $(\pi^i)^2$ form a set of orthogonal projectors with sum $1$.
\end{prop}

\begin{proof} It is similar to that of Theorem \ref{t4} c), but a little more elaborate.  Let $a_{ij}= \pi^i\pi^j-\delta_{ij} \pi^i\in CH^{2g}_\et(A\times A)[1/p]_\tors$. For any $r$ and $l\ne p$, we have
\[\pi^r\circ a_{ij} =\hat{\pi}^r\bullet a_{ij}, \quad a_{ij}\circ \pi^r = a_{ij}\bullet \hat{\pi}^r\]
by Lemma \ref{l5}, and
\[ a_{ij}\bullet \hat{\pi}^r= \hat{\pi}^{r-1}\bullet a_{ij}\]
 by Lemma \ref{l6}.
 
Consider first the case $i=j$. Then $a_{ii}$ commutes with $\pi^i$, hence  $\hat{\pi}^i\bullet a_{ii}=\hat{\pi}^{i-1}\bullet a_{ii}$, which implies
\begin{equation}\label{eq.721}
\hat{\pi}^i\bullet a_{ii}= \hat{\pi}^{i-1}\bullet a_{ii}=0
\end{equation}
hence $\pi^i\circ a_{ii}= \pi^{i-1}\circ a_{ii}=0$ and in particular
\begin{equation}\label{eq.727}
(\pi^i)^3=(\pi^i)^2 \text{ and } (\pi^i)^4=(\pi^i)^3=(\pi^i)^2. 
\end{equation}

Suppose now that $i\ne j$. Then
\begin{align*}
(\pi^i)^2\circ \pi^j &=(\pi^i+a_{ii})\circ \pi^j =a_{ij} + a_{ii}\bullet \hat{\pi}^j = a_{ij} + \hat{\pi}^{j-1}\bullet a_{ii}\\
&= \pi^i\circ (\pi^i\circ \pi^j) =\hat{\pi}^i\bullet a_{ij}
\end{align*}
i.e.
\begin{equation}\label{eq.723}
(1-\hat{\pi}^i)\bullet a_{ij}+ \hat{\pi}^{j-1}\bullet  a_{ii}=0
\end{equation}
and playing with the equality $\pi^i\circ (\pi^j)^2= (\pi^i\circ \pi^j)\circ \pi^j$, we obtain similarly
\begin{equation}\label{eq.724}
(1-\hat{\pi}^{j-1})\bullet a_{ij}+ \hat{\pi}^i\bullet  a_{jj}=0.
\end{equation}

Applying $\hat{\pi}^r$ to \eqref{eq.723} or \eqref{eq.724}, we get $\hat{\pi}^r\bullet a_{ij}=0$ for $r\ne i,j-1$, hence
\begin{equation}\label{eq.725}
 a_{ij}=  \hat{\pi}^i\bullet  a_{ij}+ \hat{\pi}^{j-1}\bullet  a_{ij}= - \hat{\pi}^i\bullet  a_{jj}- \hat{\pi}^{j-1}\bullet  a_{ii}.
\end{equation}

This in turn implies, for $i\ne j$ 
\[(\pi^i)^2\circ \pi^j = -\hat{\pi}^i\bullet a_{jj} = -\pi^i\circ ((\pi^j)^2 - \pi^j)= \pi^i\circ \pi^j - \pi^i\circ (\pi^j)^2.\]

Composing with $\pi^i$ on the left and using \eqref{eq.727}, this gives
\begin{equation}\label{eq.728}
(\pi^i)^2\circ (\pi^j)^2=0.
\end{equation}

We have proven that the $(\pi^i)^2$ form a set of orthogonal projectors; it remains to see that their sum is $1$. But
\[1=(\sum_i \pi^i)^2=\sum_i (\pi^i)^2 +\sum_{i\ne j} a_{ij}\]
which shows that $x=\sum_{i\ne j} a_{ij}$ is idempotent. But $x\circ x=0$ by Corollary \ref{c2}, so $x=0$ and the proof is complete.
\end{proof}

\begin{cor} \label{c14} Let $(\pi^i)_{i=0}^{2g}$ be such that $\sum \pi^i=1$ and $n_A^*\circ \pi^i = n^i\pi^i$ for all $i$ and all $n\in \Z$. Then the $(\pi^i)^2$ form a set of étale integral DM projectors.
\end{cor}

\begin{proof} By \cite[Th. 3.1]{dm},  the condition $n_A^*\circ \bar \pi^i = n^i\bar \pi^i$ implies that the $\bar \pi_i$ are a set of orthogonal projectors.
\end{proof}

\subsection{A weak version of Theorem \ref{t3} b)}

\begin{lemma}\label{l16} Let $(\pi^i_A)$ be a system of projectors as in Theorem \ref{t4}. Then we have
\begin{equation}\label{eq18}
2\pi^i_A\circ n_A^* =2n_A^*\circ \pi^i_A
\end{equation}
for any $n\in \Z$.
\end{lemma}

\begin{proof} We proceed as in the proof of \cite[Th. 3.1]{dm}: by Theorem \ref{t4}, we have
\[n_A^* = \sum_j n_A^*\circ \pi^j_A = \sum_j (n^j \pi^j_A + g_j(n))\]
with $2g_j(n)=0$. Moreover, $g_j(n)\circ \pi^j_A = g_j(n)$. Then
\[\pi^i_A\circ n_A^* =  \sum_j (n^j \pi^i_A\circ \pi^j_A + \pi^i_A\circ g_j(n))=n^i \pi^i_A + g_{i+1}(n)=n_A^*\circ \pi^i_A -g_i(n) + g_{i+1}(n)
\]
where we used Lemma \ref{l6}.  Hence \eqref{eq18}.
\end{proof}

\begin{thm}\label{t4a} Let $A,B$ be two abelian varieties, and let $(\pi^i_A)$, $(\pi^i_B)$ be systems of projectors as in Theorem \ref{t4}. Let $f:A\to B$ be a homomorphism. Then
\[4f^*\circ \pi^i_B = 4\pi^i_A \circ f^*\] 
for all $i$.
\end{thm}

(We cannot get less than a factor 4 in view of the indeterminacy in Theorem \ref{t4} b): take $A=B$, $f=1_A$.)

\begin{proof} We imitate the proof of \cite[Prop. 3.3]{dm}. Consider $c_{i,j} = \pi^j_A\circ f^*\circ \pi^i_B$. For $i\ne j$, $c_{i,j} $ is torsion by that proposition, hence  $0$ unless $j=i-1$ by Lemma \ref{l6}. In this case,
 we have
\begin{multline*}
n^i c_{i,i-1} = \pi^{i-1}_A\circ f^*\circ n^i \pi^i_B=\pi^{i-1}_A\circ f^*\circ (n_B^*\circ \pi^i_B+h_i(n))\\
=\pi^{i-1}_A\circ n_A^*\circ f^*\circ  \pi^i_B+\pi^{i-1}_A\circ f^*\circ h_i(n)\\
 = (n^{i-1} \pi^{i-1}_A + g_{i}(n)) \circ f^*\circ  \pi^i_B+\pi^{i-1}_A\circ f^*\circ h_i(n)\\
= n^{i-1} c_{i,i-1} + g_{i}(n) \circ f^*\circ  \pi^i_B+\pi^{i-1}_A\circ f^*\circ h_i(n)
\end{multline*}
where $h_i$ is a similar function for $B$ as $g_i$ is for $A$ in the proof of Lemma \ref{l16}, hence, taking the gcd over $n$ (or just $n=-1$),
\[4c_{i,i-1} = 0.\]

From the identity
\[f=\sum_{i,j} c_{i,j} =\sum_i (c_{i,i} +c_{i,i-1})\]
it follows that
\begin{equation}\label{eq10}
\pi_A^i\circ f^* - f^*\circ \pi_B^i = c_{i+1,i} - c_{i,i-1}
\end{equation}
is $4$-torsion.
\end{proof}

To this, we add

\begin{prop}\label{p10} For $A,B\in \Ab$, $\sM_\et(h^i(A),h^j(B))$ is torsion-free if $j\ne i-1$ (in particular for $i=j$), and  
\[\sM_\et(h^i(A),h^{i-1}(B))_\tors\simeq \Hom_{\hat{\Z}}(\hat{H}^i(A),\hat{H}^{i-1}(B))\otimes \Q/\Z.\]
\end{prop}

\begin{proof} Same as for Theorem \ref{t4} c): we have
\[\sM_\et(h^i(A),h^j(B))=\{f\in CH^{g_A}_\et(A\times B)[1/p]\mid f= \pi^j_B\circ f = f\circ \pi^i_A\}\]
and $\sM_\et(h^i(A),h^j(B))$ is torsion for $i\ne j$ by \cite{dm}. But if $f\in \sM_\et(h^i(A),h^j(B))$ is torsion, then $f\circ \pi^i_A=\pi^{i-1}_B\circ f$  by Lemma \ref{l6}, hence $f=\pi^j_B\circ \pi^{i-1}_B\circ f=0$ if $j\ne i-1$. If $j=i-1$, the conclusion follows as usual from Lemma \ref{l6} and Proposition \ref{p4} a).
\end{proof}

\subsection{A weak version of Theorem \ref{t3} a) (II); self-conjugate projectors} We first give a proof of \cite[§3, Rem. 3)]{dm}:

\begin{lemma}\label{l9} In $CH^g(A\times A)\otimes \Q$,\\
a) The operator $n_A^*$ is invertible for any $n\ne 0$.\\
b) We have
\begin{equation}\label{eq13}
\pi^i_A\circ (n_A)_* = n^{2g-i}\pi^i_A
\end{equation}
for all $(n,i)$.\\
c) We have
\begin{equation}\label{eq14}
{}^t \pi^{2g-i}_A = \pi^{i}_A
\end{equation}
for all $i$.
\end{lemma}

\begin{proof} a) holds thanks to \cite[Cor. 3.2]{dm}. Since $(n_A)_* n_A^*=n^{2g}$, \eqref{eq13} becomes true after composing with $n_A^*$ on the right, hence b) follows from a). This implies
\[n_A^*{}^t\pi^i_A = n^{2g-i}\, {}^t\pi^i_A\, \forall (n,i),\]
hence c) by the uniqueness  part of \cite[Lemma 5.4]{jannsen2}.
\end{proof}

\begin{cor}\label{c5}  In $CH^g_\et(A\times A)[1/p]$,\\
a) The identity \eqref{eq13} remains true after multiplication by $2$. \\
b) In Theorem \ref{t4} a), one can choose the $\pi^i_A$ self-conjugate.\\
c) Let $B$ be another abelian variety, and  let $(\pi^i_A)$ $(\pi^i_B)$ be two systems of self-conjugate projectors verifying Theorem \ref{t4} a). Let $f:A\to B$ be a homomorphism. Then
\[4 \pi^i_B\circ f_* = 4f_*\circ \pi^i_A\] 
for all $i$.
\end{cor}

\begin{proof} a) For $n\in \Z-\{0\}$,
\[g^i(n) = \pi^i_A\circ( (n_A)_* - n^{2g-i})\]
is torsion by Lemma \ref{l9} b), and we have
\[ g^i(mn) = n^{2g-i} g^i(m)+ g^i(n)\circ (m_A)_*\]
because $(m_A)_*$ and $(n_A)_*$ commute. But, by the same reasoning as in the proof of Lemma \ref{l9} b), $(m_A)_*$ acts as multiplication by $m^{2g-i+1}$ on $\hat{H}^{i-1}(A)\otimes \Q$, hence on $\hat{H}^{i-1}(A)$ and finally on $\hat{H}^{i-1}(A)\otimes \Q/\Z$. It follows that $g^i(n)\circ (m_A)_*=m^{2g-i+1} g^i(n)$ and the above identity becomes
\[ g^i(mn) = n^{2g-i} g^i(m)+ m^{2g-i+1}g^i(n).\]

Applying Proposition \ref{l7}, there exists $b$ such that $b= \pi^i\circ b$ and 
\begin{equation}\label{eq23}
2g^i(n)= (n^{2g-i} - n^{2g-i+1} ) b 
\end{equation}
for any $n$, so that
\[(2\pi^i_A - b) \circ (n_A)_* =  n^{2g-i}(2\pi^i_A - b) .\]

Composing with $n_A^*$ on the right, we get
\[n^{2g}(2\pi^i_A - b) =  n^{2g-i}(2\pi^i_A\circ n_A^* - b\circ n_A^*) = n^{2g}2 \pi^i_A- n^{2g-i}b\circ \pi_A^*\]
by Lemma \ref{l16}, i.e.
\[n^{2g}b =  n^{2g-i}b\circ n_A^*.\]

But since $b\in \pi^i_A\circ CH^g_\et(A\times A)[1/p]_\tors=CH^g_\et(A\times A)[1/p]_\tors\bullet \hat{\pi}^{i+1}$ (Lemmas \ref{l5} and \ref{l6}), we have $b\circ n_A^*= n^{i+1} b$, hence $(n^{2g}-n^{2g+1})b=0$ for all $n$,  which finally yields $2b=0$. But then $2g^i(n)=0$ by \eqref{eq23}.

b) By Theorem \ref{t4} b) and  Lemma \ref{l9} c), there exists $x\in CH^g_\et(A\times A)[1/p]_\tors$ such that
\begin{equation}\label{eq15} {}^t \pi^{2g-i}_A = (1+x)\pi^{i}_A(1+x)^{-1}= \pi^i_A + x\pi^i_A-\pi^i_A x=\pi^i_A + x_{i-1} - x_i
\end{equation}
for all $i$, where $x_i=\pi^i_A\circ  x$ (here we used Lemma \ref{l6}).

We first bound the order of $x$. Composing \eqref{eq15} with $2n_A^*$ on the left and using a) and the identities of Theorem \ref{t4} a), we get
\[ 2n^i\, {}^t \pi^{2g-i}_A = 2(n^i\pi^i_A + n^{i-1}x_{i-1} -n^i x_i)
\]
which, compared with \eqref{eq15} multiplied by $2n^i$, yields
\[2(n^i-n^{i-1})x_{i-1}=0.\]

Taking $n=-1$, we find $4x_{i-1}=0$ for all $i$, hence $4x=0$.

Now taking the transpose of \eqref{eq15} after switching $i$ and $2g-i$, we get
\begin{multline*}
\pi^i_A = (1+{}^t x)^{-1}{}^t \pi^{2g-i}_A(1+{}^t x)=(1+{}^t x)^{-1}(1+x)\pi^{i}_A(1+x)^{-1}(1+{}^t x)
\end{multline*}
for all $i$. 

By Theorem \ref{t4} c), this implies $(1+{}^t x)^{-1}(1+x)=1$, i.e. ${}^t x=x$, hence 
\begin{multline*}x_i = \pi^i_A \circ x = \pi^i_A\circ  {}^t x = {}^t(x\circ {}^t \pi^i_A) =  {}^t(x\bullet {}^t \hat{\pi}^i_A)=   {}^t(x\bullet \hat{\pi}^{2g-i}_A)={}^t( \hat{\pi}^{2g-i-1}_A\bullet x)\\
={}^tx_{2g-i-1}.
\end{multline*} 
 
 Let $y=\sum_{i=0}^{g-1} x_i$; noting that $x_{2g}=0$, we have  $x={}^t y+y$, or equivalently $1+x={}^t (1+y) (1+y)$. Therefore, if $\tilde \pi^i_A = (1+y) \pi^i_A (1+y)^{-1}$, we have ${}^t \tilde\pi^i_A = \tilde\pi^{2g-i}_A$ for all $i$, and $4y=0$. Then $y\circ \pi^1_A=y\bullet \hat{\pi}^1_A = \hat{\pi}^0_A\bullet y = x_0$ is not necessarily $0$, but to achieve this we just replace $y$ by $y-x_0 +{}^t x_0$, or equivalently redefine $y$ as $\sum_{i=1}^{g-1} x_i+ x_{2g-1}$.  Then $(\tilde \pi^i_A)$ still verifies Theorem \ref{t4} a), by part b) of this theorem.

c) follows from taking the transpose of the formula in Theorem \ref{t4a}.
\end{proof}

\section{Proof of Corollary \ref{c4}} 

We use the projector $\pi^1_A$ of Theorem \ref{t4}. It defines a motive $h^1_\et(A)\in \sM_\et$, direct summand of $h_\et(A)$. We write $i_A:h^1_\et(A)\to h_\et(A)$ for the inclusion and $p_A:h_\et(A)\to h^1_\et(A)$ for the projection, so that $p_A i_A=1$ and $i_A p_A = \pi^1_A$.

Let $f:A\to B$ be a homomorphism. We define
\[h^1_\et(f) = p_A f^* i_B\]
where $f^*=h(f):h_\et(B)\to h_\et(A)$.

To show that $h^1_\et$ is a (contravariant) functor, let $g:B\to C$ be another morphism. We must show that $h^1_\et(g\circ f) = h^1_\et(f)\circ h^1_\et(g)$. It suffices to prove this equality after composing with $\iota_A$ on the right and with $p_C$ on the left, hence to prove tie identity
\[\pi^1_A \circ f^* \circ g^* \circ \pi^1_C =  \pi^1_A \circ f^*\circ \pi^1_B \circ g^* \circ \pi^1_C.\]

The difference between the two sides is
\begin{multline*}
\pi^1_A \circ (f^*-f^*\circ \pi^1_B) \circ g^* \circ \pi^1_C=\pi^1_A \circ (\pi^1_A\circ f^*-f^*\circ \pi^1_B) \circ g^* \circ \pi^1_C\\
=\pi^1_A \circ (c_{2,1} - c_{1,0}) \circ g^* \circ \pi^1_C= (c_{2,1} - c_{1,0})\circ \pi^2_B  \circ g^* \circ \pi^1_C\\
= (c_{2,1} - c_{1,0})  \circ (g^*\circ \pi^2_C+d_{3,2} -d_{2,1}) \circ \pi^1_C\\
= (c_{2,1} - c_{1,0})  \circ (d_{3,2} -d_{2,1}) \circ \pi^1_C=0
\end{multline*}
where $c_{2,1} - c_{1,0}$ is as in \eqref{eq10}, $d_{3,2} -d_{2,1}$ is similar with respect to $g$ and the final vanishing follows from Corollary \ref{c2}. This proves the existence of the functor, which is obviously additive. Its full faithfulness will be proven in Corollary  \ref{c11}.

For later use, we record here:

\begin{prop}\label{p13}
We have an isomorphism
\[A(k)[1/p]_\tors\iso \sM_\et(h^1_\et(A),h_\et(X))_\tors\]
for any connected smooth projective $k$-variety $X$. 
\end{prop}

\begin{proof}  We have $\sM_\et(h_\et(A),h_\et(X))=CH^g_\et(A\times X)$. So this follows from Proposition \ref{p4a} b) and the isomorphism \eqref{eq21a}.
\end{proof}

\section{Getting rid of $2$-torsion: proof  of Theorem \ref{t3}}

\subsection{Good abelian schemes}

\begin{defn}\label{d3} Let $S$ be a smooth $k$-scheme. An abelian scheme $\sA\to S$ is \emph{good} (resp. \emph{étale-good}) if it admits a set  of integral (resp. étale integral) DM projectors (Definition \ref{d4}).  
\end{defn}

\begin{lemma}\label{l12} Suppose that $S=\Spec k$. Then\\ 
a) The sets of \'etale integral DM projectors $(\pi^i_\sA)$ form a torsor for conjugation by elements of the form $1+x$ with $2x=0$ and $x\circ \pi^1=0$.\\
b) They commute with $n_\sA^*$ for any $n\in \Z$, and  verify Theorem \ref{t4a} with multiplying by $2$ instead of $4$.\\
c) For any such set, we have the identity 
\[\pi^i_\sA\circ (n_\sA)_* = n^{2g-i} \pi^i_\sA\]
for any $i\in [0,2g]$ and any $n\in \Z$ except perhaps if $n\equiv 2 \pmod{4}$, in which case it holds at least after multiplication by $2$.
\end{lemma}

\begin{proof} a) follows from Theorem \ref{t4} c), and implies b) as one checks by going through the proofs of Lemma \ref{l16} and Theorem \ref{t4a} (the functions $g_i$ and $h_i$ used there are $0$ in the present case).

For c), recall from Corollary \ref{c5} a) that $2g^i(n)=0$ for all $n$, where
\[g^i(n) = \pi^i_\sA\circ( (n_\sA)_* - n^{2g-i}).\]

By b), we have $g^i(n)\circ n_\sA^*=0$. But, as in the proof of Corollary \ref{c5} a) for $b$, $g^i(n)\circ n_\sA^*= n^{i+1} g^i(n)$. This shows that $g^i(n)=0$ for $n$ odd; also, $g^i(n)=0$ for $n$ divisible by $4$ by Proposition \ref{l7} b).
\end{proof}

\begin{rk}\label{r3} Still by Proposition \ref{l7} b), $g^i(2m)=g^i(2)$ if $m$ is odd. So the defect in Lemma \ref{l12} c) is entirely controlled by $2_\sA$. Moreover, $g^i(2)$ does not change when one modifies the $\pi^i_\sA$ by conjugation as in Lemma \ref{l12} a), as an immediate computation shows; therefore these are invariants of $\sA$.
\end{rk}

\begin{lemma}\label{l17} Suppose that $S=\Spec k$. If $\sA$ is étale good, it enjoys a set of self-conjugate étale integral DM projectors if and only if the invariants $g^i(2)$ of Remark \ref{r3} vanish. In general, we can at least achieve $2\pi^{2g-i}_\sA = 2{}^t \pi^i_\sA$.
\end{lemma}

\begin{proof} It improves that of Corollary \ref{c5}: 

Let $(\pi^i_\sA)_{i=0}^{2g}$ be a set of étale integral DM projectors. By Lemma \ref{l12} c) and the hypothesis, $({}^t\pi^{2g-i}_\sA)_{i=0}^{2g}$ is another set of étale integral DM projectors. Reasoning as in the proof of Corollary \ref{c5} b), we find $y$ with $2y=0$ such that the projectors
\[\tilde \pi^i_\sA =(1+y)\pi^i_\sA(1+y)^{-1}\]
are self-conjugate. But the $\tilde \pi^i_\sA$ are also DM projectors, thanks to Lemma \ref{l12} a).  Conversely, if $\sA$ has a set of self-conjugate étale integral DM projectors, the invariants of Remark \ref{r3} must obviously vanish.
\end{proof}

\begin{defn}\label{d5} Under the conditions of Lemma \ref{l17}, we say that $\sA\to \Spec k$ is \emph{very good}.
\end{defn}

\begin{lemma}\label{p1}  If $\sA,\sB$ are good, then $\sA\times_S\sB$ is good;  similarly for ``étale-good'' and ``very good''. The converse is true for étale-good if $S=\Spec k$.
\end{lemma}

\begin{proof}  The `if' part of is obvious by taking the usual choice for the CK projectors
\[\pi^i_{\sA\times_S \sB} = \sum_{j+k=i} \pi^j_\sA\times_S \pi^k_\sB.\]

Conversely, let $(\pi^*_{\sA\times \sB})$ be a set of étale integral DM projectors (here we assume $S=\Spec k$, $k$ separably closed). Let $\iota:\sA\inj \sA\times \sB$ be the inclusion $a\mapsto (a,0)$, and let $\pi:\sA\times \sB\to \sA$ be the projection. Let
\[\pi^i = \iota^*\circ \pi^i_{\sA\times \sB} \circ \pi^*.\]

Then $\pi^i$ lifts the $i$-th $l$-adic Künneth projector of $\sA$ for all $l\ne p$. By Proposition \ref{l15}, the $\pi^i_A=(\pi^i)^2$ form a set of orthogonal projectors with sum $1$. Moreover, for $n\in \Z$ one has
\[n_A^*\circ \pi^i = \iota^*\circ  n_B^*\circ \pi^i_{\sA\times \sB} \circ \pi^* =\iota^*\circ  n^i\circ \pi^i_{\sA\times \sB} \circ \pi^*=n^i\pi^i\]
hence $n_A^*\circ \pi^i_A = n^i\pi^i_A$ as well.
\end{proof}

\subsection{Elliptic curves}

Let $C$ be a curve over $k$, which is supposed here to be the algebraically closed, $c\in C(k)$, and let $X$ be a smooth projective $k$-variety. The standard choice of Chow-Künneth projectors 
\[\pi^0 = \{c\}\times C, \pi^2= C\times \{c\}, \pi^1 = \Delta_C - \pi^0-\pi^2\]
acts on the group of correspondences $CH^1(C\times X)$ by composition on the right. On the other hand, we have the decomposition
\[CH^1(C\times X) = p_1^*CH^1(C)\oplus \Corr((C,c),(X,x)) \oplus p_2^*CH^1(X)\]
attached to a point $x\in X(k)$ as in \cite[\S 6, Remark]{mumford}, where $p_i$ is the $i$-th projection and $\Corr((C,c),(X,x))$ is the group of divisorial correspondences. Write $(p^0,p^1,p^2)$ for the orthogonal projectors on $CH^1(C\times X)$ with images these summands. More specifically, let $p:C\to \Spec k$ be the structural morphism and $i:\Spec k\to C$ be the section at $c$; let $q:X\to \Spec k$ be the structural morphism and $j:\Spec k\to X$ be the section at $x$. We have $p_2=p\times 1_X$, $p_1=1_C\times q$ and 
\[p^0 = (1_C\times (jq))^*,\quad p^2 = ((ip)\times 1_X)^*.\]

We have a further decomposition
\[p^0=p^0_1+p^0_2\]
where $\IM p^0_1=p_1^*\Pic^0(C)$ and $\IM p^0_2=\Z p_1^*[c]$.

\begin{lemma}\label{l14} a) For $\xi\in CH^1(C\times X)$, we have $\xi\circ \pi^2=p^2(\xi)$ and $\xi\circ \pi^0= p^0_2(\xi)$, hence $\xi\circ \pi^1=p^1(\xi)+p^0_1(\xi)$.\\
b) If $f:X\to C$ is a morphism, then
\[p^0_1({}^t\Gamma_f)=p_1^*([c]-[f(x)]).\]
\end{lemma}

\begin{proof}  We have
\[\xi\circ \pi^2=((ip)\times 1_X)^*\xi, \quad \xi\circ \pi^0=((ip)\times 1_X)_*\xi.\]

The first equality gives the first claim of a), which also implies $p^2(\xi)\circ \pi^0=0$. Then
\begin{multline*}
p^0(\xi)\circ \pi^0=((ip)\times 1_X)_*(1_C\times (jq))^*\xi=(i\times 1)_*q^*p_* (1\times j)^*\xi\\
=(i\times 1)_*q^*j^*(p\times 1)_*\xi=(i\times 1)_*(p\times 1)_*\xi=\xi\circ \pi^0
\end{multline*}
hence $p^1(\xi)\circ \pi^0=0$. Finally, note that
\[(p\times 1_X)_*\xi= d_1(\xi) [X]\]
(which defines $d_1(\xi)$), hence $ \xi\circ \pi^0=d_1(\xi) [c\times X]$. If $\xi=p_1^*\eta$, then $d_1(\xi)=\deg(\eta)$; in particular, $ p_1^*\eta\circ \pi^0=0$ if $\eta\in \Pic^0(C)$. This proves the second claim of a), hence its third claim. We then compute 
\[p^0({}^t\Gamma_f)=(1\times q)^*(1\times j)^*{}^t\Gamma_f=(1\times q)^*[f(x))]=p_1^*[f(x))]\]
hence b).
\end{proof}

\begin{thm}\label{t6} Elliptic curves $\sE\to S$ are (very) good. More precisely, the standard choice of projectors: $\pi^0_\sE = \{0\}\times_S \sE$, $\pi^2_\sE= \sE\times_S \{0\}$ and $\pi^1_\sE = \Delta_\sE - \pi^0_\sE-\pi^2_\sE$ is a set of (self-conjugate) integral DM projectors.
\end{thm}

\begin{proof} We proceed in three steps.

\subsubsection*{Step 1} $S=\Spec k$ with $k$ algebraically closed.

Obviously, $n_\sE^*\circ \pi^0_\sE = (1\times n_\sE)^*\{0\}\times \sE=\{0\}\times \sE$.

Next, $n_\sE^*\circ \pi^2_\sE = (1\times n_\sE)^*\sE\times \{0\}= \sE\times {}_n \sE$. We reduce to two cases:

\begin{description}
\item[$n$ is invertible in $k$] by the theorem of the square, the divisor class $[{}_n \sE] - n^2[0]$ equals $n^2([\sum_{x\in {}_n \sE} x]-[0])$. But  $\sum_{x\in {}_n \sE(k)} x=0$ since ${}_n \sE(k)$ is not cyclic, and $n_\sE^*\circ \pi^2_\sE=n^2\pi^2_\sE$.
\item[$n=p=\car k$] Here the group scheme ${}_n \sE$ is nonreduced. According as $\sE$ is supersingular or not, it is supported by $0$ with multiplicity $p^2$, or is of the form $F\times G$ with $F\simeq \Z/p$ and $G\simeq \mu_p$ \cite[\S 15, p. 247]{mumford}. In the first case, we obviously have $[{}_n \sE]= p^2[0]$, while in the second, we have
\[[{}_n \sE]-n^2[0]=p[F]-p^2[0]=p[\sum_{x\in F(k)} x]-p^2[0]=0\]
because $\sum_{x\in \Z/p} x$ equals $0$ when $p$ is odd and $1$ when $p=2$, but $2([1]-[0])=0$ in this case (phew!)
\end{description}

There remains the case of $\pi^1_\sE$. For this, we apply Lemma \ref{l14} with $(C,c)=(X,x)=(\sE,0)$ and $f=n_\sE$. Since $n_\sE(0)=0$, we find that $\IM \pi^1_\sE= \Corr((\sE,0),(\sE,0))$. But the map $\Corr((\sE,0),(\sE,0))\to \End(\sE)$ given by the action of divisorial correspondences is an isomorphism of rings; therefore $n_\sE^*\circ \pi^1_\sE=n\pi^1_\sE$.

\subsubsection*{Step 2} $S=\Spec k$ with $k$ separably closed. We simply use the fact that $\Pic(\sE\times_k\sE)\to \Pic((\sE\times_k\sE)\times_k \bar k)$ is injective for an algebraic closure $\bar k/k$.

\subsubsection*{Step 3} The general case. Let $\eta$ be the generic point of $S$. We have an exact sequence
\[0\to \Pic(S)\by{p_{(2)}^*} \Pic(\sE\times_S\sE)\to \Pic(\sE_\eta\times_\eta\sE_\eta)\to 0\]
where $p_{(2)}:\sE\times_S\sE\to S$ is the structural map.

Let $n$ be an integer $\ne 0$. For $i=0$, the identity $n_\sE^*\circ \pi^0_\sE=\pi^0_\sE$ is obvious. Suppose that $i=1$. By Step 2 applied over $\eta$, $n_\sE^*\circ \pi^1_\sE-n\pi^1_\sE\in \Pic(S)$. 

Let $x\in \Pic(S)$ and $\gamma\in \Pic(\sE\times_S\sE)$. An easy computation shows that
\[p_{(2)}^*x \circ \gamma =  d_1(\gamma) p_{(2)}^* x\]
where $d_1(\gamma)$ is given by the equality $(p_1)_*\gamma = d_1(\gamma)[\sE]$ in $CH^0(\sE) $ for $p_1:\sE\times_S \sE\to S$ the first projection.
In particular, $p_{(2)}^*x \circ \pi^1_\sE=0$ and $n_\sE^*\circ \pi^1_\sE-n\pi^1_\sE=(n_\sE^*\circ \pi^1_\sE-n\pi^1_\sE)\circ \pi^1_\sE=0$. 

For $\pi^2$, we have to see that $[{}_n \sE] = n^2[0]\in \Pic(\sE)$. The argument is the same as over an algebraically closed field, since the theorem of the cube still holds \cite[\S 10]{mumford}.
\end{proof}

\begin{rk}\label{r6} The proof of Theorem \ref{t6} does not require $k$ to be separably closed, hence this theorem is valid without such hypothesis.
\end{rk}

\subsection{A deformation result}

\begin{prop}\label{pDM} Let $\pi:\sA\to S$ be an abelian scheme and let $s,t\in S(k)$. If $\sA_s$ is étale-good, so is $\sA_t$, and similarly for ``very good''.
\end{prop}

\begin{proof}
If $S=\Spec k$, Step 3 in the proof of Theorem \ref{t4} a) shows that we have a collection of obstruction classes
\[ c_i(\sA)\in HH^1(\Lambda,M_i(\sA))\, (2\le i\le 2g)\]
with
\[M_i(\sA)=\Hom(H^i_2(\sA),H^{i-1}_2(\sA))\otimes \Q/\Z\] 
such that $\sA$ is étale-good if and only if $c_i(\sA)=0$ for all $i$. This extends to any abelian scheme $\pi:\sA\to S$, with obstruction classes:
\[ c_i(\sA)\in HH^1(\Lambda,CH^i_\et(\sA))[1/p]_\tors).\]

Sheafifying for the étale topology, we get local obstruction classes
\[\tilde c_i(\sA)\in H^0_\et(S, HH^1(\Lambda,R^{2i}\pi_*\Z(i))[1/p]_\tors).\]

By Proposition \ref{p4} c) and ``Gersten's principle'' \cite[2.4]{bvk}, the natural maps
\[R^{j-1}\pi_*\Z_l\otimes \Q_l/\Z_l(i)\to R^j\pi_*\Z(i))\{l\}\]
remain isomorphisms for $l\ne p$. By smooth and proper base change, $c_i(\sA_s)=0$ $\iff$ $\tilde c_i(\sA)=0$ $\iff$ $c_i(\sA_t)=0$. Same reasoning for the invariants of Remark \ref{r3}, hence for ``very good'' by Lemma \ref{l17}.
\end{proof}

\subsection{Proof of Theorem \ref{t3}}

 All abelian varieties are supposed to be over $k$. By Lemma \ref{p1} and  Theorem \ref{t6}, products of elliptic curves are good, hence étale-good. By the connectedness of the moduli scheme of principally polarised abelian varieties \cite[Ch. IV, Cor. 5.10]{FC} and Proposition \ref{pDM}, all principally polarised abelian varieties are étale-good. But then so is any abelian variety $A$ by Lemma \ref{p1} again, since $(A\times \hat{A})^4$ is principally polarised (Zarhin's trick). The same argument shows that all principally polarised abelian varieties are very good, since the DM projectors for elliptic curves in Theorem \ref{t6} are self-conjugate.\qed

\section{Proofs of Theorem \ref{t5} and Corollary \ref{c12}}

We give ourselves a system $(\pi^j_A)$ of projectors verifying the conditions of Theorem \ref{t4}.

\subsection{Beauville's decomposition, $p$-integrally}\label{s11.1}

\begin{lemma} \label{l8} Let $i\ge 0$. Then,\\
a) We have $\pi_A^{2i-1}x=x$ for any $x\in  CH^i_\et(A)[1/p]_\tors$.\\
b) For $j\ne 2i$, $\pi^j_A CH^i_\et(A)[1/p]$ is divisible.\\
c) For $j\ne 2i-1$, $\pi^j_A CH^i_\et(A)[1/p]$ is torsion-free.
 \end{lemma}

\begin{proof} a) write $x=\sum_j x^j$, with $x^j=\pi^j_A x$, and let  $n\in \Z$. Then each $x^j$ is torsion hence, by Proposition \ref{p4} c), we have $n_A^*x^j =n^{2i-1} x^j$ for any $j$. On the other hand, $2n_A^*x^j = 2n^j x^j$. Therefore, $2(n^j - n^{2i-1}) x^j= 0$. For $j\ne 2i-1$, this means that $\pi^j_A CH^i_\et(A)[1/p]_\tors$ has finite exponent. But then it is $0$, as a quotient of a divisible group.

b) Let $N=\bigcap_{l\ne p}\Ker \cl_l$; we have $N=\bigoplus_{j=0}^{2g} \pi^j_A N$. Since $N$ is divisible by Proposition \ref{p4a} a), all these summands are divisible. But $\pi^j_A CH^i_\et(A)[1/p])\subset N$ if $j\ne 2i$, hence $\pi^j_A CH^i_\et(A)[1/p])= \pi^j_A N$ for those $j$.

c) follows from a). 
\end{proof}

In order to respect Beauville's numbering, we set

\begin{defn}\label{d1} $CH^i_{\et,s}(A) = \pi^{2i-s}_A CH^i_\et(A)[1/p]$.
\end{defn}

Note that $CH^i_\et(A)[1/p]=\bigoplus_{s=2(i-g)}^{2i} CH^i_{\et,s}(A)$.

\begin{prop}\label{p7} Let $B_s=\{x\in CH^i_\et(A)[1/p]\mid n_A^* x = n^{2i-s} x\, \forall n\in \Z\}$. Then\\
a) For $s\ne 0$ we have
\[B_s= F_s\oplus CH^i_{\et,s}(A)  \]
where $F_1=0$ and $F_s={}_{w_{2i-s,2i-1}} CH^i_\et(A)[1/p]$ for $s\ne 1$.\footnote{Here, we write by convention $w_{a,b} := w_{b,a}$ if $a<b$.}\\
b) For $s=0$, we have inclusions
\[2CH^i_{\et,0}(A)\oplus {}_2CH^i_\et(A)[1/p]\subseteq B_0\subseteq CH^i_{\et,0}(A)\oplus {}_2CH^i_\et(A)[1/p].\]
If $(\pi^j_A)$ is a system of étale integral DM projectors, we even have
\begin{equation}\label{eq26}
B_0= CH^i_{\et,0}(A)\oplus {}_2CH^i_\et(A)[1/p].
\end{equation}
\end{prop}

Note that in a) and b), the sums are direct since $CH^i_{\et,s}(A)$ is torsion-free by  Lemma \ref{l8} c).

\begin{proof} If $x\in B_s$, let $x_t=\pi^{2i-t}_A x$ for all $t$. Let $n\in \Z$. On the one hand, $2n_A^* x_t = 2n^{2i-t} x_t$ by Theorem \ref{t4} a). On the other hand, $2n_A^* x_t = 2\pi^{2i-t}_A n_A^*x$ by Lemma \ref{l16}.  So $2n_A^* x_t = 2n^{2i-s}x_t$ and $ 2(n^{2i-t}- n^{2i-s}) x_t=0$, thus $x_t$ is of torsion bounded independently of $x$ if $t\ne s$. Thus there is an integer $P>0$ such that $P(x-x_s)=0$ for all $x\in B_s$, i.e. $B_s\subseteq CH^i_{\et,s}(A)+{}_PCH^i_\et(A)[1/p] $. We also have $2CH^i_{\et,s}(A) \allowbreak\subseteq B_s$  by Theorem \ref{t4} a).

Suppose that $s\ne 0$; then Lemma \ref{l8} b) implies that $CH^i_{\et,s}(A) \allowbreak\subseteq B_s$ and that this inclusion is split. Therefore $B_s=  CH^i_{\et,s}(A)\oplus F_s$ where $F_s$ is of finite exponent. But $F_s\subset CH^i_{\et,1}(A)$ by Lemma \ref{l8} a); this shows that $F_1=0$ and that, for $s\ne 1$, $F_s$ is killed by $w_{2i-s,2i-1}$.  Conversely, ${}_{w_{2i-s,2i-1}} CH^i_\et(A)[1/p]$ is clearly contained in $B_s$. This proves a), and b) is obtained by the same reasonings: more precisely, the integer $P$ of the beginning may be taken as $w_{2i,2i-1}=2$. 

In the case of a system of étale integral DM projectors, the inclusion $2CH^i_{\et,0}(A) \allowbreak\subseteq B_0$ improves to $CH^i_{\et,0}(A) \allowbreak\subseteq B_0$, hence the last claim. 
\end{proof}

\begin{cor}\label{c13} a) The subgroup $CH^i_{\et,s}(A)$ of $CH^i_\et(A)[1/p]$ does not depend on the choice of $(\pi^i_A)$ if $s\ne 0$. The same is true for the subgroup $CH^i_{\et,0}(A)\oplus {}_2CH^i_\et(A)[1/p]$ in the case of étale integral DM projectors.\\
b) Let $(\tilde \pi^j_A)$ be another set of  projectors verifying Theorem \ref{t4}; for $x\in CH^i_\et(A)[1/p]$ and $s$, let $x_s=\pi^{2i-s} x$ and $\tilde x_s = \tilde \pi^{2i-s}_A x$. Then $x_s=\tilde x_s$ for $s\ne 0,1$ and $\tilde x_0-x_0=x_1-\tilde x_1$ belongs to ${}_4CH^i_\et(A)[1/p]$, and even to ${}_2CH^i_\et(A)[1/p]$ in the case of étale integral DM projectors.
\end{cor}

\begin{proof} a) Indeed, $CH^i_{\et,s}(A)$ is the maximal divisible subgroup of $B_s$ for $s\ne 0$ by Lemma \ref{l8} b) and Proposition \ref{p7} a). For $s=0$, this follows directly from \eqref{eq26}.

b) For any $s$, the operator $\tilde\pi^{2i-s} - \pi^{2i-s}$ is torsion by Theorem \ref{t4} b); therefore it vanishes on the maximal divisible subgroup of  $CH^i_\et(A)[1/p]$, hence on $CH^i_{\et,t}(A)$ for $t\ne 0$. On the other hand, for $s\ne 0,1$, $x_s$ and $\tilde x_s$ belong to a common uniquely divisible subgroup by a), hence so does their difference which is therefore $0$.   This gives the first claim, which implies $x_0+x_1=\tilde x_0+\tilde x_1$. But $4(\tilde \pi^{2i} - \pi^{2i})=0$, and even $2(\tilde \pi^{2i} - \pi^{2i})=0$ if $(\pi^j)$ and $(\tilde \pi^j)$ are étale integral DM projectors.
\end{proof}

\begin{cor}\label{c6} We have $CH^i_{\et,s}(A) =0$ if $s\notin [i-g,i]$.
\end{cor}

\begin{proof} By \cite[Théorème]{beauville2}, $B_s\otimes \Q$ vanishes in this range. Therefore so does $CH^i_{\et,s}(A)\otimes \Q$ by Proposition \ref{p7}, and also $CH^i_{\et,s}(A)$ by Lemma \ref{l8} c). (Note that $1\in [i-g,i]$ except for $i=0$ and that $CH^0_{\et,1}(A)=0$ trivially.)
\end{proof}

\begin{cor}\label{c6a} We have $CH^1_{\et,1}(A) =\Pic^0(A)$. 
\end{cor}

\begin{proof} We have $CH^1_{\et,1}(A) =B_1$ by Proposition \ref{p7} a): this is $\Pic^0(A)$ since $\NS(A)$ is torsion-free.
\end{proof}

\begin{cor}\label{c16} Let $f:A\to B$ be an isogeny. Then $f_*CH^i_{\et,s}(A)\subseteq CH^i_{\et,s}(B)$ for $s\ne 0$. If the $(\pi^i_A)$ and the $(\pi^i_B)$ are both étale integral DM projectors,  then $f_*CH^i_{\et,0}(A)\subseteq CH^i_{\et,0}(B)\oplus {}_2CH^i_\et(B)[1/p]$.
\end{cor}

\begin{proof} Since $n_B^*f_* = f_* n_A^*$ for any $n\in\Z$, we have $f_* B_s(A)\subseteq B_s(B)$ for any $s$. The conclusion follows as in the proof of Corollary \ref{c13}.
\end{proof}

For the next corollary, note that $B^i_s\cdot B^j_t\subset B^{i+j}_{s+t}$ for all $i,j,s,t$, with obvious notation.

\begin{cor}\label{c7} We have $CH^i_{\et,s}(A)\cdot CH^j_{\et,t}(A)\subseteq CH^{i+j}_{\et,s+t}(A)$ if $(s,t)\ne (0,0)$.
\end{cor}

\begin{proof} Indeed, in this case either $CH^i_{\et,s}(A)$ or $CH^j_{\et,t}(A)$ is divisible by Lemma \ref{l8} b), hence so are their tensor product and its image in  $B^{i+j}_{s+t}$. This image is therefore contained in the largest divisible subgroup of $B^{i+j}_{s+t}$, which is contained in $CH^{i+j}_{\et,s+t}(A)$ by Proposition \ref{p7} (even if $s+t=0$).
\end{proof}

\subsection{The good divided powers}\label{s11.2}

Let $x\in CH^i_\et(A)[1/p]$ with $i>0$, and let $n\ge 2$. We first define $\gamma_n(x)$ when $x\in CH^i_{\et,s}(A) $ for some $s$. Suppose $s\ne 0$. Then $x^n\in CH^{ni}_{\et,ns}(A) $ by Corollary \ref{c7}. Since $ns\ne 0,1$, this group is uniquely divisible by Lemma \ref{l8} and we just set $\gamma_n(x) = \frac{x^n}{n!}$.

Suppose now that $s=0$. Then $x^n= x^{[n]} + y_n $, with $x^{[n]} \in CH^{ni}_{\et,0}(A)$ and $2y_n=0$ (unique decomposition). In the torsion-free group $CH^{ni}_{\et,0}(A)$, $x^{[n]}$ is divisible by $n!$ (same reasoning as in the proof of Proposition \ref{t1}, using Corollary \ref{c8}). We define
\[\gamma_n(x) = \frac{x^{[n]}}{n!}.\]

In general, write $x=\sum_s x_s$ with $x_s\in CH^i_{\et,s}(A) $. We set
\[\gamma_T(x,s) = \sum_{n\ge 0} \gamma_n(x_s)T^n\in CH^*_\et(A)[1/p][T]\]
where $\gamma_0(x_s):=1$ and $\gamma_1(x_s):=x_s$, 
\[\gamma_T(x) = \prod_s \gamma_T(x,s)\]
and $\gamma_n(x)=$ the $n$-th coefficient of $\gamma_T(x)$. 

We may extend the $\gamma_n$ by the same trick to inhomogeneous elements of $CH^{>0}_\et(A)[1/p]$ (this will not be used in the rest of this article).

\begin{rk} This construction depends on the choice of $(\pi^i_A)$, and does not satisfy the divided power identities in $CH^{>0}_\et(A)[1/p]$. To get them, we must go modulo $2$-torsion as in the next subsection.
\end{rk}

\subsection{\'Etale motives modulo $2$-torsion} We introduce here a category which will help formulate our results:

\begin{defn}\label{d6} For an abelian variety $A$ and $i\ge 0$, we set
\[\oCH^i_\et(A) = CH^i_\et(A)[1/p]/{}_2CH^i_\et(A)[1/p].\]
If $B$ is another abelian variety of dimension $g_B$, we set
\[\oCorr_\et(A,B) = \oCH^{g_B}_\et(B\times A).\]
We write $\bar \sM_\et^\ab$ for the pseudo-abelian hull of the additive category with objects abelian varieties and morphisms the $\overline{\Corr}_\et(A,B)$, and $\bar h_\et(A)$ for the image of $A$ in $\bar \sM_\et^\ab$.
\end{defn}

Note that the functors $\hat{R}:\sM^\ab_\et\to \hat{\sM}^\ab$ and $R_B:\sM^\ab_\et\to \sM_B^\ab$ both factor through $\bar \sM_\et^\ab$. One could equally define a larger category $\bar \sM_\et$ encompassing all smooth projective varieties, but the above would fail in general, so this does not seem too useful.

We can now reformulate some of the previous results as follows:

\begin{thm}\label{t7} For any abelian variety $A$ of dimension $g$,\\ 
a) A system of self-conjugate projectors $(\pi^i_A)$ as in Corollary \ref{c5} b) induces a system of self-conjugate projectors in $\oCorr_\et(A,A)$ having the property of \cite{dm} with respect to multiplications; we call it a \emph{system of DM projectors modulo $2$-torsion}. We have
\[\pi^i_A\circ n_A^* = n_A^*\pi^i_A,\quad \pi^i_A\circ (n_A)_* = n^{2g-i}\pi^i_A
\]
for any $i$ and any $n\in \Z$.\\ 
b) For $i\ge 0$, the direct sum decomposition
\[\oCH^i_\et(A) =\bigoplus_s \oCH^i_{\et,s}(A)\]
where $\oCH^i_{\et,s}(A)$ is the image of $CH^i_{\et,s}(A)$ in $\oCH^i_\et(A)$, does not depend on the choice of $(\pi^i_A)$; $\oCH^i_{\et,s}(A)$ is divisible for $s\ne 0$ and torsion-free for $s\ne 1$.\\
c) If $(\pi^i_A)$ is étale integral (Definition \ref{d3}), its image in $\oCorr(A,A)$ is self-conjugate and does not depend on the choice; we call it the \emph{canonical system of DM projectors} for $A$.\\ 
d) If $B$ is another abelian variety provided with a system $(\pi^i_B)$ as in a) and if $f:A\to B$ is a homomorphism, we have
\[ 2f^*\circ \pi^i_B  =2 \pi^i_A\circ f^*\in \oCorr(B,A), \quad 2\pi^{2g_B-i}_B\circ f_* = 2f_*\circ \pi^{2g_A-i}_A\in \oCorr(A, B)\] 
for all $i$, where $g_A$ and $g_B$ are the dimensions of $A$ and $B$. If the systems $(\pi^i_A)$ and $(\pi^i_B)$ are canonical, we can remove the factor $2$ in both identities; in particular, $f^*\oCH^i_{\et,s}(B)\subseteq \oCH^i_{\et,s}(A)$ and  $f_*\oCH^{g_A-i}_{\et,s}(A)\subseteq \oCH^{g_B-i}_{\et,s}(B)$ for any $(i,s)$.
\end{thm}

\begin{proof} a) follows from Lemma \ref{l16} and Corollary \ref{c5} a). b) follows from Corollary \ref{c13} a).  c) follows from Lemmas \ref{l12} and \ref{l17}. d)  follows from   Theorem \ref{t4a}, Corollary \ref{c5} c) and Lemma \ref{l12} b).
\end{proof}

We now push $\gamma_T(x)$ in $\oCH^*_\et(A)[T]$. We have:

\begin{prop}\label{p11}  The function $\gamma_T:CH^{>0}_\et(A)[1/p]\to \oCH^*_\et(A)[T]$ factors through $\oCH^{>0}_\et(A)$.
\end{prop}

\begin{proof}  Let $i>0$ and $x,y\in CH^i_\et(A)[1/p]$, with $2y=0$: we must show that $2\gamma_T(x+y) = 2\gamma_T(x)$. Recall that $y\in CH^i_{\et,1}(A)$ (Lemma \ref{l8} a)). This already shows that $\gamma_T(x,s)=\gamma_T(x+y,s)$ for $s\ne 1$; but for $s=1$ and $n>1$,
\[(x_1+y)^n = x_1^n +nyx_1^{n-1} = x_1^n\]
where the first equality comes from Proposition \ref{p4} d) and the second one from the divisibility of $CH^{ni}_{\et,n}(A)$ (Lemma \ref{l8} b)); which implies that $x_1$ is divisible by $2$. Therefore $\gamma_n(x_1+y) = \gamma_n(x_1)$.
\end{proof}

\begin{proof}[Proof of Theorem \ref{t5}] a) The divided power identities of the introduction are all homogeneous of a certain degree $N$: $N=0,1$ in (1), $N=n$ in (2) and (3), $N=m+n$ in (4) and $N=mn$ is (5). 

They are true tautologically after tensoring with $\Q$, as well as in degree $N=1$. For the identities of degree $N>1$, the target group is torsion free (in $\oCH^*_\et(A)$!), as seen in the definition of the $\gamma_n$'s,  so they remain true.

b) Same argument as in a).
\end{proof}

\begin{defn}\label{d8}
Although the   $\gamma_i$'s now verify the divided power identities, they still depend on the choice of $(\pi^j_A)$ in $\oCH^{>0}_\et(A)$. If we choose the canonical system of DM projectors as in Theorem \ref{t7} c), we shall talk of the \emph{canonical divided powers}.
\end{defn}

\subsection{Proof of Corollary \ref{c12}}\label{s11.3} From now on, all abelian varieties are provided with their canonical  divided powers (Definition \ref{d8}).

For $x\in \oCH^{>0}_\et(A)$, set
\[e^x = \sum_n \gamma_n(x).\]

This is the polynomial $\gamma_T(x)$ evaluated at $T=1$. By Property (3) of divided powers, we have the identity
\begin{equation}\label{eq27}
e^{x+y} = e^x e^y.
\end{equation}

\begin{defn} Let $\ell\in \Pic(A\times \hat{A})$ be the class of the (normalised) Poincaré bundle. We define $\sF_A$ as the (inhomogeneous) correspondence given by $e^\ell$.
\end{defn}

To prove Corollary \ref{c12} (i) and (ii), we want to follow the arguments of \cite[Proof of Proposition 3']{beauville}. However, this proof rests on Corollary 2 to Proposition 3 in the said article. Therefore, we must first prove an integral version of that corollary.

\begin{prop}\label{p18} Let $\pi:A\times \hat{A}\to A$ be the first projection. Then   $\pi_*(e^\ell)=[0]\in \oCH^g_\et(A)$.
\end{prop}

\begin{proof} We have to show that
\[\pi_*\gamma_i(\ell) =
\begin{cases}
0&\text{if } i\ne g\\
[0]&\text{if } i= g,
\end{cases}
\]
this being known after tensoring with $\Q$ by \cite[Cor. 2 to Prop. 3]{beauville}. 
Since $\ell\in \oCH^1_{\et,0}(A\times \hat{A})$, $\gamma_i(\ell)\in \oCH^i_{\et,0}(A\times \hat{A})$ by construction. By Theorem \ref{t7} d), $\pi_*\gamma_i(\ell)\in \oCH^{i-g}_{\et,0}(A)$. Since this group is torsion-free, we win.
\end{proof}

Thanks to Proposition \ref{p18} and its dual (transpose $e^\ell$), Corollary \ref{c12} (i) and (ii) are proven as said before, by using \eqref{eq27} and Theorem \ref{t5} a).  Moreover, (iii) is given by the argument of Mukai \cite[proof of (3.4)]{mukai} and Theorem \ref{t5} b).

\begin{rk} If we relax the condition of taking the canonical divided powers, Proposition \ref{p18}, and therefore Corollary \ref{c12}, remain true after multiplying by $2$.
\end{rk}

\section{Further properties of the Fourier transform}\label{s11.5} 

\subsection{DM projectors} Let $(\pi^i_A)$ be a system of DM projectors on $A$ modulo $2$-torsion (see Theorem \ref{t7} a)). We then get a system of projectors in $\oCH^g_\et(\hat{A}\times \hat{A})$
\[\pi^i_{\hat{A}}=\sF_A\circ \pi^{2g-i}_A\circ \sF_A^{-1}.\]

\begin{prop}\label{p15} The $\pi^i_{\hat{A}}$ also form a system of DM projectors modulo $2$-torsion.
\end{prop}

\begin{proof} This follows from the commutation of $\sF_A$ and $\sF_{\hat{A}}$ with isogenies and from Corollary \ref{c5} a).
\end{proof}

\begin{rk} It seems reasonable to expect that $(\pi^i_{\hat{A}})$ is canonical if $(\pi^i_A)$ is (see Theorem \ref{t7} c)), but I don't see a way to prove it.
\end{rk}

\begin{cor}[cf. {\cite[Prop. 1]{beauville2}}]\label{p14} a) The Fourier transforms $\sF_A$ and $\sF_{\hat{A}}$ exchange $\oCH^i_{\et,s}(A)$ and $\oCH^{g-i+s}_{\et,s}(\hat{A})$ for all $(i,s)$.\\
b) The Pontryagin divided powers $\gamma_n^*$ of Corollary \ref{c8a} respect these subgroups.
\end{cor}

\begin{proof} a) follows from Proposition \ref{p15} and Theorem \ref{t7} b); b) follows from a) and the same property for the $\gamma_n$.
\end{proof}

\begin{cor}[cf. {\cite[Cor. 1 to Prop. 3]{beauville}}]\label{c15} Let $c\in \oCH^i_\et(A)$ be a symmetric (resp. antisymmetric) element. Then the component of $\sF_A(c)$ on $\oCH^{g-i+r}_\et(A)$ is $0$ if $r$ is odd (resp. even).
\end{cor}

\begin{proof} Indeed, $c$ belongs to $\bigoplus\limits_{s \text{ even}} \oCH^i_{\et,s}(A)$ (resp. to $\bigoplus\limits_{s \text{ odd}} \oCH^i_{\et,s}(A)$).
\end{proof}

\subsection{Beauville's identities}

\begin{prop}\label{p16} Suppose that $A$ is provided with a polarisation $\phi:A\to \hat{A}$ of degree $\nu$, induced by the class $d\in CH^1(A)$ of an ample symmetric divisor. Then we have, for any $i\in [0,g]$
\[\nu\sF_A(\gamma_i(d)) = (-1)^{g-i}\phi_*(\gamma_{g-1}(d)),\quad \nu\gamma_i^*(c) = \nu^i\gamma_{g-i}(d)\]
where $c=\gamma_{g-1}(d)$. In particular,
\begin{equation}\label{eq25}
\gamma_g(d)=\nu [0].
\end{equation}
\end{prop}

\begin{proof} Same as for Proposition \ref{p18}: the identities hold after tensoring with $\Q$ by \cite[Prop. 6 and Cor. 2]{beauville}, but both of their terms are respectively in the torsion-free groups $\oCH^{g-i}_{\et,0}(\hat{A})$ and $\oCH^i_{\et,0}(A)$ (for the first identity, by Theorem \ref{t7} d) and Corollary \ref{p14} a); for the second one, by Corollary \ref{p14} b)).
\end{proof}

\subsection{Application: Scholl's formula for the DM projectors}

\begin{prop}\label{p17} In  Proposition \ref{p16}, suppose that the polarisation $\phi$ is principal. Consider the correspondences
\begin{equation}\label{eqscholl}
p_i= (-1)^i \sum_{2a+b= 2g-i \text{ and } b+2c=i} \gamma_a(\pi_1^*d)\cdot \gamma_b((1\times \phi)^*\ell)\cdot \gamma_c(\pi_2^* d)
\end{equation}
in $\oCH^g_\et(A\times A)$, where $\pi_1,\pi_2$ are the two projections $A\times A\to A$ (recall that $\ell$ is the class of the normalised Poincaré bundle). Let $\pi^i_A=p_i\circ p_i$. Then $(\pi^i_A)$ is a system of DM projectors modulo $2$-torsion.
\end{prop}

\begin{proof} By \cite[5.9]{scholl}, the statement is true after tensoring with $\Q$. Moreover, \eqref{eq25} implies as in loc. cit. that $\sum p_i=1$ (this is the key point). We note that Proposition \ref{l15} works equally well in $\oCH^g_\et(A\times A)$. Hence the claim.
\end{proof}

\begin{rk} I cannot extend Proposition \ref{p17} to a non-principal polarisation as in \cite{scholl}, except of course up to multiplying by its degree. I am also not sure what condition is needed on $d\in CH^1(A)$ to ensure that $(\pi^i_A)$ is canonical.
\end{rk}

\subsection{Jacobians of curves: Suh's formula} Let $C$ be a (smooth projective) curve of genus $g$, with Jacobian $J$. Choose a rational point on $c$ and, for every $i\in [0,g]$, let $W^{[i]}=[C^{(g-i)}]\in CH^i(J)$ be the corresponding cycle class (image of $S^{g-i}(C)$ by the canonical map). 

\begin{thm}[\'Etale integral Poincaré formula]\label{t8} a) We have $W^{[i]}=\gamma_i(d)$ in $\oCH^i_\et(J)$, where $d=W^{[1]}$. \\
b) We have $W^{[r]}\cdot W^{[s]} = \binom{r+s}{r} W^{[r+s]}$ for any $(r,s)$.
\end{thm}

\begin{proof} a) We first prove this after tensoring with $\Q$. Then $W^{[i]}=\gamma_{g-i}^*(c)$ where $c= [C]$ \cite[(1.2.1]{mo-po}, and the result follows from \cite[Cor. 2 to Prop. 6]{beauville} (the rational version of the second identity in Proposition \ref{p16}). But both terms belong to the torsion-free group $\oCH^i_{\et,0}(J)$: for 
$\gamma_i(d)$ this is because $d\in \oCH^1_{\et,0}(J)$, and for  $W^{[i]}$ this is because $n_A^*W^{[i]}=n^{2i}W^{[i]}$ for all $n\in \Z$, and $\oCH^i_{\et,0}(J)= \overline{B}^i_0(J)$ with obvious notation (which follows from \eqref{eq26}). Thus the formula holds integrally.

b) This follows from a) and the identity (4) in the introduction for divided powers.
\end{proof}

The following corollary is an étale integral analogue (modulo $2$-torsion) of J. Suh's formula for the Künneth projectors of $J$ \cite[Th. 4.2.3]{suh}:

\begin{cor}\label{q8} Consider the classes 
\begin{multline}\label{eqsuhalg}
P_i=(-1)^i\sum_{2a+b=2g-i \text{ and } b+2c=i} p_1^*W^{[a]}\cdot p_2^*W^{[c]}\\
\times \sum_{d+e+f=b}(-1)^{d+f} p_1^*W^{[d]}\cdot \mu^*W^{[e]}\cdot p_2^*W^{[f]}\in \Corr(J, J).
\end{multline}
Then the $\pi^i_A:= P_i\circ P_i$ form a system of DM projectors modulo $2$-torsion in $\oCorr_\et(J,J)$.
\end{cor}

\begin{proof} I claim that $P_i=p_i$, where $p_i$ is as in \eqref{eqscholl}. The argument is the same as in \cite{suh}: in \eqref{eqsuhalg}, the term indexed by $(d,e,f)$ is an expansion of the term $\gamma_b((1\times \phi)^*\ell)$ in \eqref{eqscholl} given in Proposition \ref{p17}, thanks to Theorem \ref{t8} a) and the identity $(1\times \phi)^*\ell=p_1^*d +p_2^*d - \mu^*d$. 
\end{proof}

\begin{rk} Another formula for $P_i$ in $\oCorr_\et(J,J)$ is
\begin{equation}\label{eqsuh2}
P_i=\sum_{r,s} (-1)^{g-r-s}\binom{r+s}{s+i-g}  p_1^*W^{[r]}\cdot\mu^*W^{[g-r-s]}\cdot   p_2^*W^{[s]}.
\end{equation}

This is proven by using Theorem \ref{t8} b) and the Chu-Vandermonde identity; details are left
 to the reader. 
 \end{rk}

\section{Proof of Theorem \ref{c1}} \label{s9}

Let $S$ be an essentially smooth $k$-scheme, where $k$ is now an arbitrary field.\footnote{$k$ could even be a Dedekind domain, see \cite[Rem. 1.1]{kun-ar}, but one would have to invert all primes $p$ that are not invertible in $k$. This somewhat reduces the interest of this generalisation when $k$ is a localised ring of integers in a number field.} We want to extend what comes before to abelian $S$-schemes $\sA$ in the style of \cite{dm}. This is not straightforward, as the torsion of $CH^g_\et(\sA\times_S \sA)$ is more complicated.  So we proceed step by step.

\subsection{An arithmetic computation}

\begin{nota}\label{n1} Let $A,B>0$ be two integers. We write $A\sim B$ if $A$ and $B$ have the same prime factors.
\end{nota}

\begin{lemma}\label{l2} For $i,c\ge 0$, let
\[N(i,c)= \prod_{a\le c}w_{i,i-a-1}\]
where $w_{i,j}$ is as in Definition \ref{eq16}. Then $N(i,c)\sim\inf(i,c+2)!$.
\end{lemma}

\begin{proof} Let $l$ be prime. By Lemma \ref{l11} and Definition \ref{eq16}, $w_{i,j}=1$ for $j\le 0$ and $l\mid w_{i,j}$ for $i>j>0$ $\iff$ $l-1\mid i-j$. Thus $l\mid N(i,c)$ $\iff$ $\exists\ a\le \inf(i-2,c): l-1\mid a+1$ $\iff$ $l-1\le \inf(i-1,c+1)$ $\iff$ $l\le \inf(i,c+2)$ $\iff$ $l\mid \inf(i,c+2)!$.
\end{proof}

\subsection{DM projectors}\label{s12.1}

\begin{lemma}\label{l19} For any $i,j\ge 0$, the map
\begin{equation}\label{eq24}
H^{j}_\cont(\sA,\hat{\Z}(i))\to H^0_\et(S,R^{j}\pi_* \hat{\Z}(i))
\end{equation}
 has cokernel  killed by 
 \[N(j):=\prod_{r=2}^{\cd(S)} w_{j,j-r+1}=N(j,\cd(S)-2)\sim \inf(j,\cd(S))!.\]
\end{lemma}

\begin{proof} Consider the Leray spectral sequence for continuous étale cohomology. A weight argument using multiplications on $\sA$ shows that it degenerates up to torsion; more precisely,  the differential $d_r^{a,b}$ is killed by $w_{b,b-r+1}$ for all $r$. Applying this to $(a,b)=(0,j)$, we get the claim. The equivalence $\sim$ follows from Lemma \ref{l2}.
\end{proof}

\begin{defn}\label{d7}
Let $T$ be a (not necessarily additive) functor from the category of abelian $S$-schemes to an additive category $\sC$. There is a homomorphism of monoids
\[T_\sA:\End_S(\sA)\to \End_\sC(T(\sA)).\]
We write $\End_\sC(T(\sA))^c$ for the centraliser of $T_\sA(\Z 1_\sA)$ in $\End_\sC(T(\sA))$.
\end{defn}

Take for $T$ the composition of the motive functor and the composition \eqref{eq29} (over the base $S$). The resulting endomorphism ring is $H^{2g}_\cont(\sA\times_S\sA,\hat{\Z}(g))$, and we write $H^{2g}_\cont(\sA\times_S\sA,\hat{\Z}(g))^c$ for the corresponding subring.

\begin{prop} \label{p22} The restriction to $H^{2g}_\cont(\sA\times_S\sA,\hat{\Z}(g))^c$ of the edge homomorphism
\[H^{2g}_\cont(\sA\times_S\sA,\hat{\Z}(g))\to H^0_\et(S,R^{2g}(\pi\times_S\pi)_* \hat{\Z}(g))\]
of the Leray spectral sequence
\[E_2^{a,b}=H^a_\et(S,R^b(\pi\times_S\pi)_* \hat{\Z}(g))\Rightarrow H^{a+b}_\et(\sA\times_S \sA, \hat{\Z}(g))\] 
becomes bijective after inverting $N'=N(2g)$ (cf. Lemma \ref{l19}).
\end{prop}

\begin{proof} By smooth and proper base change, the sheaves $R^b(\pi\times_S\pi)_*  \hat{\Z}(g)$ are locally constant; therefore a similar computation to that in \S \ref{s3} holds, namely
\[R^b(\pi\times_S\pi)_* \hat{\Z}(g)\simeq \bigoplus_i \uHom(R^i\pi_*\hat{\Z},R^{i+b-2g}\hat{\Z})\]
hence a decomposition $E_2^{a,b}= \bigoplus_i E_2^{a,b}(i)$, with
\[E_2^{a,b}(i)=H^a_\et(S,\uHom(R^i\pi_*\hat{\Z},R^{i+b-2g}\hat{\Z}).\]

The monoid $\End_S(\sA)$ acts on the spectral sequence. Let $m\in \Z$. The above shows that $m_\sA$ acts on $E_2^{a,b}(i)$ on the right by multiplication by $m^i$ and on the left by multiplication by $m^{i+b-2g}$.

Let $\phi: H^{2g}_\cont(\sA\times_S\sA,\hat{\Z}(g))^c\to H^0_\et(S,R^{2g}(\pi\times_S\pi)_* \hat{\Z}(g))$ be the morphism under consideration. The above computation shows inductively that $\Ker \phi$ is killed by $N'$.

To handle $\Coker \phi$, consider for all $(i,j)$ the subgroup 
\[M(i,j)=\{x\in H^{2g}_\cont(\sA\times_S\sA,\hat{\Z}(g))\mid \forall\ m\in\Z, xm_\sA^*=m^ix \text{ and } m_\sA^*x =m^jx\}.\]

From (the proof of) Lemma \ref{l19}, we deduce that  the Leray spectral sequence degenerates after inverting $N'$, and from this and the previous computation we deduce that
\[H^{2g}_\cont(\sA\times_S\sA,\hat{\Z}(g))[1/N']= \bigoplus_{i,j} M(i,j)[1/N'].\]

Since $M(i,j)$ maps to $0$ in $H^0_\et(S,R^{2g}(\pi\times_S\pi)_* \hat{\Z}(g))$ for $i\ne j$, it follows that $\phi$ becomes surjective after inverting $N'$.
\end{proof}

\begin{rk} Getting an explicit upper bound for the exponent of $\Coker \phi$ is left to the courageous reader.
\end{rk}

\begin{cor}\label{c19} For $i\in [0,2g]$, let $\Pi^i_\sA\in CH^g(\sA\times_S\sA)\otimes \Q$ be the $i$-th projector of Deninger-Murre. Then its cycle class $\cl(\Pi^i_\sA)\in H^{2g}_\cont(\sA\times_S\sA,\hat{\Z}(g))\otimes \Q$ is in the image of $H^{2g}_\cont(\sA\times_S\sA,\hat{\Z}(g))[1/N']\to H^{2g}_\cont(\sA\times_S\sA,\hat{\Z}(g))\otimes \Q$.
\end{cor}

\begin{proof} Indeed, $\cl(\Pi^i_\sA)\in H^{2g}_\cont(\sA\times_S\sA,\hat{\Z}(g))^c\otimes \Q$, and its image in \break$H^0_\cont(S,R^{2g}(\pi\times_S\pi)_* \hat{\Z}(g))\otimes \Q$ lies in $H^0_\cont(S,R^{2g}(\pi\times_S\pi)_* \hat{\Z}(g))\subset H^{2g}_\cont(A\times A,\hat{\Z}(g))$, where $A$ is the geometric generic fibre of $\sA$. The corollary therefore follows from the proposition.
\end{proof}

\begin{lemma}\label{l10} The ideal $CH^g_\et(\sA\times_S \sA)[1/p]_\tors$ of $CH^g_\et(\sA\times_S \sA)[1/p]$ is nilpotent.
\end{lemma}

\begin{proof} Consider the Leray spectral sequence
\begin{equation}\label{eq28}
H^a_\et(S,R^b(\pi\times_S\pi)_* \Z(g))\Rightarrow H^{a+b}_\et(\sA\times_S \sA,\Z(g))
\end{equation}
where $\pi:\sA\to S$ is the projection. It induces a filtration $F^aCH^g_\et(\sA\times_S \sA)[1/p]$ on $CH^g_\et(\sA\times_S \sA)[1/p]$ of length $\le \inf(\cd(S),2g)$, and it is standard that $F^a\circ F^{a'}\subseteq F^{a+a'}$. So it suffices to show that $CH^g_\et(\sA\times_S \sA)[1/p]_\tors\circ CH^g_\et(\sA\times_S \sA)[1/p]_\tors\subseteq F^1$. But this follows from Corollary \ref{c2}, applied over a separable closure of $k(S)$.
\end{proof}

\begin{proof}[Proof of Theorem \ref{c1} a)]
We mimick the proof of Theorem \ref{t4}, starting from the DM projectors of \cite{dm}.  
Corollary \ref{c19} alllows us to get to Step 1, and Lemma \ref{l10} then allows us to reach Step 2. Thus we get CK projectors after inverting $pN'$. For Step 3, we consider this time the Leray spectral sequence
\[H^a_\et(S,R^b(\pi\times_S\pi)_* \Q_l/\Z_l(g))\Rightarrow H^{a+b}_\et(\sA\times_S \sA,\Q_l/\Z_l(g)).\]

Still by smooth and proper base change, we get from Proposition \ref{p4} c) isomorphisms
\[(R^b(\pi\times_S\pi)_* \Z_l(g))\otimes \Q_l/\Z_l\iso R^b(\pi\times_S\pi)_* \Q_l/\Z_l(g).\]

If $(\hat{\pi}^i(\sA))_{i=0}^{2g}$ denote the K\"unneth projectors in $H^0(S,R^{2g}(\pi\times_S\pi)_* \hat{\Z}(g))[1/N']$ from Corollary \ref{c19}, then, by Lemma \ref{l6}, $\hat{\pi}^i(\sA)$ acts on the right on $R^b(\pi\times_S\pi)_* \hat{\Z}(g)$ as $\hat{\pi}^{i+b-2g}(\sA)$ acts on the left, and the same holds for their action on $H^a_\et(S,R^b (\pi\times_S\pi)_* (\Q/\Z)'(g))$ for all $a$. For $a+b=2g$, this group is $0$ for $a> \inf(\cd(S),2g)$.

We show the statement modulo $F^a CH^g_\et(\sA\times_S \sA)[1/N'p]_\tors:= F^a\cap CH^g_\et(\sA\times_S \sA)[1/N'p]_\tors$, by induction on $a$. Let $(\pi^i_\sA(a-1))_{i=0}^{2g}$ be a set of DM projectors modulo $F^{a-1}$, and $(\tilde \pi^i_\sA(a-1))$ a lift modulo $F^a$. As in Step 3 of the proof of Theorem \ref{t4} a), the obstruction for $\pi^i_\sA(a)$ is a $1$-cocycle $f$ with values in $\gr^a$, verifying
\[f(mn) = m^{i-a-1}f(n) + n^if(m).\]

(If $i-a-1\le 0$, there is no obstruction.) By Proposition \ref{l7}, $w_{i,i-a-1} f$ is a coboundary. Noting that
\[\prod_{i=0}^{2g} N(i,\inf(2g,\cd(S)))\sim \prod_{i=0}^{2g} \inf(i,\cd(S)+2)!\sim \inf(2g,\cd(S)+2)!,\]
we see that the primes dividing the integer $N'$ of Proposition \ref{p22} also divide this product.
This shows that the DM projectors exist after inverting $\inf(2g,\cd(S)+2)!$.

We reason in the same way for uniqueness. Let $(\tilde\pi^i)$ be another set of DM projectors; by Step 2, there exists $z\in CH^g_\et(\sA\times_S \sA)[1/p]_\tors$ such that 
\begin{equation}\label{eq30}
\tilde \pi^i=(1+z)\pi^i_\sA(1+z)^{-1}
\end{equation}
for all $i$. We show that $z\equiv 0\pmod{F^{a} CH^g_\et(\sA\times_S \sA)[1/N'p]_\tors}$ for all $a$, by induction on the latter. Assume this is proven modulo $F^{a-1}$, for $a>0$; in particular, $z^2\equiv 0\pmod{F^a}$. Let  $\bar z$ be the image of $z$ in $\gr^{a} CH^g_\et(\sA\times_S \sA)[1/N'p]_\tors$.   Reasoning as in the proof of Theorem \ref{t4} b), we have $z\pi^i\equiv \pi^{i-a}z\pmod{F^a}$, which yields the identity
\[(n^i-n^{i-a})\bar z\pi^i=0\ \forall\ n\]
by applying $n^*_\sA$ to \eqref{eq30} on the left. Since $w_{i,i-a}$ is invertible in $\Z[1/N'p]$, we get $\bar z\pi^i=0$ for all $i$, hence $\bar z=0$.
\end{proof}

\begin{rk}\label{r4} If $\cd(S)<2$, then $N'=1$ in Proposition \ref{p22} and we get CK projectors on the nose; if $\cd(S)=0$, they moreover have the properties of Theorem \ref{t4} by the same arguments. In the general case, by chasing denominators one can get $p$-integral operators $\Pi^i$ and an integer $M_1$ with the same prime divisors $l$ as $\inf(2g,\cd(S)+2)!$ such that $\Pi^i\circ \Pi^j= M_1\delta_{ij}\Pi^i$ and $M_1n^*\Pi^i = M_1n^i\Pi^i$ for all $n\in \Z$; the present method does not yield an explicit upper bound for $v_l(M_1)$, however.     \end{rk}

\begin{rk}\label{r7} Suppose that $\sA\simeq \sA'\times_S \sA''$ for abelian schemes $\sA',\sA''$ of relative dimensions $g',g''$. From DM projectors for $\sA'$ and $\sA''$; one gets DM projectors for $\sA$ by the usual Künneth formula as in the proof of Lemma \ref{p1}. Thus one may replace $g$ with $\sup(g',g')$ in $\inf(2g,\cd(S)+2)!$. When inverting this smaller constant, one may lose uniqueness, however.
\end{rk}

\subsection{Divided powers}

\begin{lemma}\label{l13} For any $i\ge 0$, the group $\pi^{2i}_\sA CH^i_\et(\sA)[1/M!p]$ is torsion-free.
\end{lemma}

\begin{proof} The torsion subgroup of $CH^i_\et(\sA)$ is the image of $H^{2i-1}_\et(\sA,\Q/\Z(i))$; therefore it suffices to show that $\pi^{2i}_\sA H^{2i-1}_\et(\sA,\Q/\Z(i)))[1/M!p]=0$. This follows from the Leray spectral sequence
\[H^a_\et(S,R^b\pi_*\Q/\Z(i)))\Rightarrow H^{a+b}_\et(\sA,\Q/\Z(i)))\]
as $\pi^{2i}_\sA R^b\pi_*\Q/\Z(i)=0$ for $b\ne 2i$, by smooth and proper base change. 
\end{proof}

\begin{proof}[Proof of Theorem \ref{c1} b)] Since the bigraded sheaf $R^*\hat\Z(*)$ is lisse, its global sections are its invariants under $\pi_1(S,\bar \eta)$ where $\bar \eta$ is a geometric generic point of $S$; since this sheaf is torsion-free, so are its global sections. It follows that the augmentation ideal of $\bigoplus_{i\ge 0} H^0_\cont(S,R^{2i}\pi_*\hat\Z(i)))$ has divided powers. Consider now the composition
\[CH^i_\et(\sA)\otimes \hat\Z\by{\cl} H^{2i}_\cont(\sA,\hat\Z(i))\to H^0_\cont(S,R^{2i}\pi_*\hat\Z(i))).\]

Inverting $M!p$ and applying $\pi^{2i}_\sA$ yields a new composition
\begin{multline}\label{eq3}
\pi^{2i}_\sA CH^i_\et(\sA)[1/M!p]\otimes \hat\Z\by{\cl} \pi^{2i}_\sA H^{2i}_\cont(\sA,\hat\Z(i))[1/M!p]\\ 
\iso H^0_\cont(S,R^{2i}\pi_*\hat\Z(i)))[1/M!p]
\end{multline}
where the isomorphism follows from the degeneration of the Leray spectral sequence (proof of Lemma \ref{l19}) and the fact that $\pi^{2i}_\sA$ is a DM projector. Note that $\xi\in \pi^{2}_\sA CH^1_\et(\sA)[1/M!p]$ by hypothesis. Since $\cl(\xi)^i$ is divisible by $i!$ in the last term of \eqref{eq3}, $\xi^i$ is divisible by $i!$ in $\pi^{2i}_\sA CH^i_\et(\sA)[1/M!p]$ by Corollary \ref{c8}, and uniquely so by Lemma \ref{l13}. The same torsion-freeness then implies the divided power identities.
\end{proof}

\subsection{The Fourier transform} In the previous subsection, replace $\sA$ with $\sB=\sA\times_S \hat{\sA}$ and take $\pi^*_\sB=\pi^*_\sA\otimes_S \pi^*_{\hat{\sA}}$ as in Remark \ref{r7}. According to this remark, we also replace $\inf(4g,\cd(S)+2)$ with $\inf(2g,\cd(S)+2)=M$. Applying Theorem \ref{c1} b) to the normalised Poincaré bundle $\ell$, we define the Fourier transform
\begin{equation}\label{eq31}
\sF_\sA = \sum_{i\ge 0} \gamma_i(\ell)\in CH^*(\sB)[1/M!p].
\end{equation}

Thanks to Lemma \ref{l13}, Proposition \ref{p18} extends to this case, hence we get the desired properties of $\sF_\sA$. This concludes the proof of Theorem \ref{c1}.

\appendix

\section{The Picard and Albanese functors}

The base field $k$ is separably closed.

\begin{defn}\label{dA1} We let $\Abs$ denote the additive category of commutative $k$-group schemes whose objects are extensions of a constant, finitely generated group scheme by an abelian variety. If $P\in \Abs$, we write $P^0$ for its identity component and $\bar P$ for $P/P^0$. We say that $P$ is \emph{discrete} if $P^0=0$.\\
Let $P,Q\in \Abs$. We say that a homomorphism $P(k)\to Q(k)$ is \emph{representable} if it is induced by a morphism in $\Abs$. 
\end{defn}

The aim of this appendix is to prove Theorems \ref{t1a} and \ref{tA1} below.

\begin{thm}\label{t1a} There exists a naturally commutative diagram of additive functors 
\[\begin{CD}
\sM@>\underline{\Pic}>> \Abs\\
@VVV @VVV\\
\sM_\et@>\underline{\Pic}_\et>> \Abs[1/p]
\end{CD}\]
such that $\underline{\Pic}(h(X)) = \underline{\Pic}_\et(h_\et(X))=\Pic_{X/k}$ for any smooth projective $X$.
\end{thm}

The functor $X\mapsto \Pic(X)=H^2_\et(X,\Z(1))$ from smooth projective varieties to abelian groups extends to a functor on $\sM$, and even on $\sM_\et$ when we invert $p$; we want to show that these functors induce the desired functors. For this, we need a few lemmas:

\begin{lemma}\label{lA1}
The subfunctor $X\mapsto \Pic^0(X)$ of $\Pic$ extends to a functor on $\sM$, and even on $\sM_\et$ after inverting $p$.
\end{lemma}

\begin{proof} If $\alpha:X\to Y$ is a Chow correspondence, then $\alpha(\Pic^0(X))\subseteq \Pic^0(Y)$ because algebraic equivalence is an adequate equivalence relation. If $\alpha$ is now an étale Chow correspondence, there is an integer $n>0$ such that $n\alpha$ is algebraic. Since $\Pic^0(Y)$ is the maximal divisible subgroup of $\Pic(Y)$, we also have $\alpha(\Pic^0(X)[1/p])\subseteq \Pic^0(Y)[1/p]$.
\end{proof}

\begin{lemma} \label{lA5} The category $\Abs$ is pseudo-abelian.
\end{lemma}

\begin{proof} Indeed, any idempotent endomorphism $e$ of $\Abs$ has a kernel, hence an image (= kernel of $1-e$).
\end{proof}

\begin{lemma}\label{lA2} Let $A,B$ be two abelian varieties, and let $\phi:A(k)\to B(k)$ be a homomorphism. If there exists an integer $n>0$ prime to $p$ such that $n\phi$ is representable, then so is $\phi$.
\end{lemma}

\begin{proof} Since $B(k)$ is Zariski dense in $B$, the homomorphism $\Abs(A,B)\to \Hom(A(k),B(k))$ is injective. The conclusion follows from a chase in the commutative diagram of exact sequences
\[\begin{CD}
0&\to& \Abs(A,B)@>n>> \Abs(A,B)@>>> \Abs({}_n A,B)\\
&& @VVV @VVV @V\wr VV\\
0&\to& \Hom(A(k),B(k))@>n>> \Hom(A(k),B(k))@>>> \Hom({}_n A,B(k))
\end{CD}\]
where the last vertical map is an isomorphism. (We used that $A$ and $A(k)$ are $n$-divisible.)
\end{proof}

\begin{lemma}\label{lA4} For any $P\in \Abs$, the inclusion $P^0\inj P$ has a retraction in $\Abs[1/p]$. 
\end{lemma}

\begin{proof} It is equivalent to show that $P\to \bar P$ has a section in the category of étale sheaves tensored with $\Z[1/p]$, and this is further equivalent to showing that $\Ext^1(\bar P, P^0)=0$ in this category. We immediately reduce to $\bar P=\Z$ or $\bar P=\Z/n$ for $n>0$ prime to $p$. In the first case, the $\Ext^1$ is $H^1_\et(k,P^0)=0$. In the second case, the exact sequence $0\to \Z\by{n}\Z\to \Z/n\to 0$ yields an exact sequence
\[P^0(k)\by{n}P^0(k)\to \Ext^1(\Z/n,P^0)\to 0\]
and we conclude since $P^0(k)$ is $n$-divisible.
\end{proof}

\begin{lemma}\label{lA3} Let $P,Q\in \Abs$ and let $\phi:P(k)\to Q(k)$ be a homomorphism. Assume that the restriction of $\phi$ to $P^0(k)$ is representable. Then $\phi$ is representable by a unique morphism $f:P\to Q$.
\end{lemma}

\begin{proof} First, notice that the representing morphism $f^0:P^0\to Q$ is unique, again because $P^0(k)$ is Zariski dense in $P^0$. Consider now the commutative diagram of exact sequences
\begin{equation}\label{eqA1}\small
\begin{CD}
0&\to& \Abs(\bar P,Q)@>>> \Abs(P,Q)@>a>> \Abs(P^0,Q)\\
&& @V\wr VV @VVV @V VV\\
0&\to& \Hom(\bar P,Q(k))@>>> \Hom(P(k),Q(k))@>>> \Hom(P^0(k),Q(k))
\end{CD}
\end{equation}
where the left vertical map is an isomorphism. By Lemma \ref{lA4},  $a$ is surjective and the claim follows from diagram chase.
\end{proof}

\begin{proof}[Proof of Theorem \ref{t1a}] We first deal with $\underline{\Pic}$. 
By Lemma \ref{lA5}, it suffices to define it on motives of smooth projective varieties. For this, we use a theorem of Friedlander-Voevodsky (cf. \cite[proof of Prop. 2.1.4]{voetri}): for two  (connected) smooth projective varieties $X,Y$, the group $CH^{d_X}(X\times Y)$ is generated by classes of finite correspondences. Let $\alpha$ be such a finite correspondence, assumed to be integral:  the map
\[\alpha:\Pic(Y)\to \Pic(X)\]
factors as a composition
\[\Pic(Y)\by{q^*}\Pic(Z)\by{p_*} \Pic(X)\]
where $p$ and $q$ are the projections of $Z$ on $X$ and $Y$ (here the existence of $p_*$ follows from the normality of $X$, \cite[Déf. 21.5.5 and (21.5.5.3)]{ega}\footnote{In loc. cit., §21.5, finite morphisms should presumably be assumed to be surjective.}). 
 By Yoneda's lemma, to show that $q^*$ and $p_*$ are representable it suffices to show that they extend to morphisms of Picard functors. This is trivial for $q^*$, while for $p_*$ it follows from \cite[Prop. 21.5.8]{ega}.
 
 More directly, thanks to the finiteness of $p$, ``Schapiro's lemma'' yields an isomorphism
\[H^1_\et(Z,\G_m)\simeq H^1_\et(X,p_*\G_m)\]
as well as a norm map $N:p_*\G_m\to \G_m$ induced by the norm on the function fields; their composition is $p_*$. This is obviously compatible with base change.

To construct $\underline{\Pic}_\et$, it suffices to prove the representability of morphisms $\alpha:h_\et(X)\to h_\et(Y)$ for $X,Y$ smooth projective. By the existence of $\underline{\Pic}$, $\Pic^0(n\alpha)$ is representable for some $n>0$, and therefore so is $\Pic(\alpha)$ by Lemmas \ref{lA2} and \ref{lA3}.
\end{proof}

In order to get a correct adjunction statement, we need to introduce a rather baroque category:

\begin{defn}\label{dA2}a) We write $\widetilde{\Abs}$ for the category where
\begin{itemize}
\item objects are triples $(P,Q,\alpha)$ where $P,Q\in \Abs$, $\bar P$ is free,  and $\alpha$ is an isomorphism $\widehat{P^0}\iso Q^0$;
\item a morphism $(P,Q,\alpha)\to (P',Q',\alpha')$ is a pair of morphisms $f:P'\to P$, $g:Q\to Q'$ such that $\alpha'\circ\widehat{f^0}= g^0\circ \alpha$.
\end{itemize}
b) Let $\Sigma$ be the set of cyclic groups of the form $\Z/l^n$ where $l$ is a prime number different from $p$, $n\ge 1$ and $\Z/l^n$ is a direct summand of $\NS(X)$ for some smooth projective $X$. We write $\widetilde{\Abs}'$ for the full subcategory of $\widetilde{\Abs}$ whose objects are the  triples $(P,Q,\alpha)$ such that any cyclic direct summand of $\bar Q_\tors[1/p]$ of prime power order belongs to $\Sigma$.
\end{defn}

The explanation of this category is given by the following definition:

\begin{defn}\label{dA3}We write $\widetilde{\Pic_\et}$ for the functor $\sM_\et\to \widetilde{\Abs}'[1/p]$ induced by the rule
\[\widetilde{\Pic_\et}(h_\et(X))=(\widetilde{\Alb}_X,\underline{\Pic}_\et(h_\et(X)),\alpha)\]
for $X$ smooth projective, where $\widetilde{\Alb}_X$ is the Albanese scheme of $X$ and $\alpha:\widehat{\Alb_X}\iso \Pic^0_X$ is the canonical isomorphism between the Picard variety of $X$ and the dual of its Albanese variety.
\end{defn}

(The existence of the component $\widetilde{\Alb}_X$ as a functor on $\sM_\et$ is shown as for $\underline{\Pic}_\et$.)

We now have the following complement toi Theorem \ref{t1a}.

\begin{thm}\label{tA1}  The functor $\widetilde{\Pic_\et}$ has a fully faithful left adjoint $h^1_+$.
\end{thm}

\begin{cor}\label{c11} The functor $h^1_\et$ of Corollary \ref{c4} is fully faithful.\qed
\end{cor}

\begin{proof}[Proof of Theorem \ref{tA1}] The category $\widetilde{\Abs}'[1/p]$ is additive; to prove the existence of $h^1_+$ it suffices  in view of Lemma \ref{lA4} to show that it is defined at objects of the form $(A,\hat{A},1_A)$, $(P,0,1)$ and $(0,Q,1)$ where $A\in \Ab$, $P$ and $Q$ are discrete and $P$ is free.

We set
\begin{thlist}
\item  $h^1_+(P,0,1) = P^*\otimes \un$ where $\un$ is the unit motive and $P^*$ is the $\Z$-dual of $P$,
\item $h^1_+(0,Q,1) = Q\otimes\L$ where $\L$ is the Lefschetz motive; this makes sense since $\sM_\et$ is pseudo-abelian,
\item $h^1_+(A,\hat{A},1_A) = h^1_\et(A)$,
\end{thlist}
and check the adjunction property in each case; it is enough to check it ``against'' motives of the form $h_\et(X)$ for $X$ smooth projective and connected.

We also have an obvious unit isomorphism 
\[\eta_{(P,Q,\alpha)}:(P,Q,\alpha)\iso \widetilde{\Pic_\et} h^1_+(P,Q,\alpha)\]
in each special case. We now check that the composition
\begin{multline*}
\rho:\sM_\et(h^1_+(P,Q,\alpha)),h_\et(X))\to  \widetilde{\Abs}'[1/p](\widetilde{\Pic_\et}h^1_+(P,Q,\alpha)),\widetilde{\Pic_\et} h_\et(X))\\
\by{\eta_{(P,Q,\alpha)}^*}\widetilde{\Abs}'[1/p]((P,Q,\alpha)),\widetilde{\Pic_\et} h_\et(X))
\end{multline*}
is an isomorphism in each case:

\subsubsection*{Case (i)} 
\begin{multline*}
\sM_\et(P^*\otimes \un,h_\et(X))=\Hom(P^*,CH^0(X))\\
=\Hom(P^*,\Z)=P=\widetilde{\Abs}'[1/p]((P,0,1),\widetilde{\Pic_\et}(h_\et(X)). 
\end{multline*}

\subsubsection*{Case (ii)} 
\begin{align*}
\sM_\et(Q\otimes \L,h_\et(X))&=\Hom(Q,CH^1(X))\\
 \widetilde{\Abs}'[1/p]((0,Q,1),\widetilde{\Pic_\et}(h_\et(X))&= \Abs(Q,\Pic_X).
 \end{align*} 

The isomorphism follows from a diagram chase similar to that in the proof of Lemma \ref{lA3}.
\subsubsection*{Case (iii)} 
\begin{align*}
\sM_\et(h^1_\et(A),h_\et(X))&=CH^g_\et(A\times X)[1/p]\circ \pi^1_A \\
 \widetilde{\Abs}'[1/p]((A,\hat{A},1_A),\widetilde{\Pic_\et}(h_\et(X))&= \Abs(\widetilde{\Alb}_X,A)
 \end{align*} 
where $g=\dim A$. The map 
\[\rho:CH^g_\et(A\times X)[1/p]\circ \pi^1_A\to \Abs(\widetilde{\Alb}_X,A)[1/p]\] 
is induced by the functor $\widetilde{\Alb}$. We have an exact sequence
\begin{equation}\label{eq22}
0\to A(k)\to \Abs(\widetilde{\Alb}_X,A)\to \Ab(\Alb_X,A)\to 0.
\end{equation}

For $l\ne p$, $\rho$ is an isomorphism on $l$-primary torsion by Proposition \ref{p13} and the isomorphism
\[ \Abs(\widetilde{\Alb}_X,A)\{l\}=A(k)\{l\}\]
stemming from \eqref{eq22}. On the other hand, $\Coker \rho$ is torsion-free, as one sees by using $\cl_l$ and Proposition \ref{p4a} a).

Suppose first that $A=J(C)$ for a curve $C$. Pick a point $x\in X(k)$. Then  Lemma \ref{l14} yields an isomorphism
\[ \sM_\et(h^1_\et(A),h_\et(X))\simeq \Pic^0(C)\oplus \Corr((C,c),(X,x))\]
which, in view of \eqref{eq22}, implies that $\rho$ is an isomorphism. In general, let $C$ be an ample curve traced on $A$. Then $A$ is almost a direct summand of $J(C)$, hence $\rho$ has torsion kernel and cokernel and is therefore an isomorphism.

The fact that the unit morphism is an isomorphism implies that $h^1_+$ is fully faithful.
\end{proof}

\section{Open questions}

\subsection{$p$-torsion in characteristic $p$}\label{sB.1}

Suppose $k$ algebraically closed. If the cohomology algebra $\bigoplus_{i\ge 0} H^{2i}_\cont(A,\Z_p(i))$ also has divided powers, we can replace $(r!)_p$ with $r!$ in Proposition \ref{t1}. If, on the other hand, one looks for a counterexample, the simplest possible instance is for $p=2$ and a divided square on $H^{2}_\cont(A,\Z_p(1))$. More specifically, one may ask the following question: 

\begin{qn}\label{q1}
Is there an abelian variety $A$ over an algebraically closed field of characteristic $2$ and a divisor $D$ on $A$ such that $x^2\ne 0$ in $H^2(A,\nu(2))$, where $x$ is the cycle class of $D$ in $H^1(A,\nu(1))$?
\end{qn}

Here, $\nu(n)=\Omega^n_{\log}$. Note that one must have $\dim A\ge 3$ in view of Ro\v\i tman's theorem, which is independent of the characteristic by  \cite{milne2}.

In general, it would be interesting to understand what happens if one does not invert $p$ (see also footnote \ref{f2} on p. \pageref{page}).

\subsection{Self-conjugate DM projectors} 

\begin{qn}\label{q6}
Can one remove the condition that $A$ is principally polarised in Theorem \ref{t3} c)? This would be possible if, in Lemma \ref{p1}, the converse statement could be extended from ``étale good'' to ``very good''. I was not able to prove this.
\end{qn}

\subsection{Pontryagin divided powers}

\begin{qn}\label{q4} In Corollary \ref{c8a}, we constructed a system of divided powers on the ring $\oCH_*^\et(A)$ provided with the Pontryagin product.  In \cite{mo-po}, the same was done on $CH_*(A)$ by a completely different (geometric) method. Does the natural map $CH_*(A)\to \oCH_*^\et(A)$ commute with divided powers?
\end{qn}

\subsection{Torsion in Néron-Severi groups}

\begin{qn}\label{q3} How large is the set $\Sigma$ in Definition \ref{dA2} b)?
\end{qn}

\subsection{\'Etale Picard motive} Let $M\in \sM_\et$. The counit of the adjunction of Theorem \ref{tA1} is a morphism
\[h^1_+(\widetilde{\Pic_\et} M)\to M\]
which yields in particular a morphism $h^1_\et(\underline{\Pic}_\et M)\to M$. 

After tensoring with $\Q$, this morphism becomes split for $M=h(X)$, $X$ smooth projective, by a Hard Lefschetz argument (see \cite[Th. 4.4 (ii)]{scholl}), and one recovers Murre's Picard motive from \cite{murre}. 

\begin{qn}\label{q5}
Does this morphism already split in $\sM_\et$?
\end{qn}

\subsection{Comparison with ordinary Chow groups}\label{sB.6} For a smooth projective $k$-variety $X$ of dimension $g$, we have the following result (\cite[Th. 3.11]{ct-v}, \cite[Cor. 4.10]{pirutka}):

\begin{prop} For any prime $l$, there is a surjection with divisible kernel
\[ H^{g-3}(X,\sH^g_\et(\Q_l/\Z_l(g-1)))\Surj C_\tors\]
where $C=\Coker(CH^{g-1}(X)\otimes \Z_l\by{\cl_l} H^{2g-2}_\cont(X,\Z_l(g-1)))$.
\end{prop}

Suppose that $X$ is an abelian variety $A$. Understanding the group $C_\tors$ is important in view of the results of \cite{BG,EGS,EGS2}. For any $n\in\Z$, $n_A^*$ acts on this group by multiplication by $n^{2g-2}$. In view of \S \S \ref{s6} and \ref{s7}, this motivates the following

\begin{qn}\label{q7} What is the action of $n_A^*$ on $H^{g-3}(A,\sH^g_\et(\Q_l/\Z_l(g-1)))$?
\end{qn}

Any information on the said action (except if it involves multiplication by $n^{2g-2}$) would bound the order of $C_\tors$.

\subsection{Jacobians of curves}\label{sB.7}

\begin{qn} Do the $P_i$ of \eqref{eqsuhalg} define integral DM projectors already in $\Corr(J,J)=CH^g(J\times J)$? Does the identity \eqref{eqsuh2} already hold in this group? (In \cite[Th. 2.3]{mo-po}, the identity of Theorem \ref{t8} b) is proven in $CH^{r+s}(J)$, modulo $2$-torsion, for certain hyperelliptic Jacobians).
\end{qn}

\end{document}